\documentclass[letterpaper,11pt]{amsart}
\usepackage{tikz, tikz-cd, stmaryrd}
\usetikzlibrary{arrows}
\usepackage{subcaption} 
\usepackage[font=footnotesize,labelfont=bf]{caption}
\usepackage{blkarray}
\usepackage{enumitem}
\usepackage{latexsym,array,delarray,epsfig,setspace,cleveref, mathtools,amssymb,mathrsfs,bbold}

\newtheorem{theorem}{Theorem}
\numberwithin{theorem}{section}
\newtheorem{proposition}[theorem]{Proposition}
\newtheorem{corollary}[theorem]{Corollary}
\newtheorem{lemma}[theorem]{Lemma}

\newtheorem{question}[theorem]{Question}
\theoremstyle{definition}
\newtheorem{definition}[theorem]{Definition}
\newtheorem{remark}[theorem]{Remark}
\newtheorem{example}[theorem]{Example}

\newcommand{\Cc}{{\mathbb C}}
\newcommand{\Nn}{{\mathbb N}}
\newcommand{\Qq}{{\mathbb Q}}
\newcommand{\Rr}{{\mathbb R}}
\newcommand{\Zz}{{\mathbb Z}}
\newcommand{\Pp}{{\mathbb P}}
\newcommand{\Aa}{{\mathbb A}}
\newcommand{\Kk}{{\mathbb K}}

\renewcommand{\AA}{\mathcal{A}}

\newcommand{\OO}{\mathcal{O}}

\newcommand{\isom}{\cong}				%% "isomorphic to"
\newcommand{\st}{~|~}					%% vertical bar for "such that"
\newcommand{\D}{\mathfrak{D}}

\DeclareMathOperator{\Spec}{Spec}           
\DeclareMathOperator{\Trop}{Trop}
\DeclareMathOperator{\Hom}{Hom}

\title{The well-poised property and torus quotients}
\author{Joseph Cummings and Christopher Manon}

\begin{document}

\maketitle

\begin{abstract}
    An embedded variety is said to be well-poised when the associated initial ideal degenerations coming from points of the tropical variety are reduced and irreducible.  Varieties with a well-poised embedding admit a large collection of explicitly constructible Newton-Okounkov bodies. 
    This paper aims to study the well-poised property under torus quotients. Our first result states that GIT quotients of normal well-poised varieties by quasi-tori also have well-poised embeddings. As an application, we show that several Hassett spaces, $\overline{M}_{0,\beta}$, are well-poised under Alexeev's embedding (see \cite{Alexeev}). Conversely, given an affine $T$-variety $X$ with polyhedral divisor $\D$ on a well-poised base $Y$, we construct an embedding of $X \subseteq \Aa^N$ and provide conditions on $Y$ and $\D$ which if met, imply $X$ is well-poised under this embedding. Then we show that any affine arrangement variety meets the specified criteria, generalizing results of Ilten and the second author for rational complexity 1 varieties. Using this result, we explicitly compute many Newton-Okounkov cones of $X$ and provide a criterion for the associated toric degenerations to be normal. Our final application combines these two results to show that hypertoric varieties have well-poised embeddings.
\end{abstract}

\tableofcontents

%%%%%%%%%%%%%%%%%%%%%%%%%%
%       INTRO
%%%%%%%%%%%%%%%%%%%%%%%%%%
\section{Introduction}

In this paper, we are interested in varieties whose tropicalizations are particularly well-behaved. We let $X = V(I) \subseteq \Aa^n$ be an affine irreducible variety over $\Cc$ which does not lie in any coordinate hyperplane. Recall that every point $w \in \Rr^n$ defines an \emph{initial ideal} $\mathrm{in}_w(I) \subseteq \Cc[x_1, \ldots, x_n]$ (see \cite{Maclagan-Sturmfels}, \cite{Cox-Little-O'Shea}, \cite{GBCP}). Following \cite{Maclagan-Sturmfels}, the \emph{tropical variety} $\Trop(I) \subseteq \Rr^n$ is defined to be the set of those $w$ for which $\mathrm{in}_w(I)$ contains no monomials. $\Trop(I)$ can be given the structure of a finite rational polyhedral fan so that if $w_1,w_2$ lie in the relative interior of a cone $C$ in $\Trop(I)$, their initial ideals coincide. We set $\mathrm{in}_C(I)$ to be $\mathrm{in}_w(I)$ for any $w$ in the relative interior of $C$.

\begin{definition}
Let $X = V(I) \subseteq \Aa^n$ be as above, and let $C$ be a cone in $\Trop(I)$. Then $C$ is a \textit{prime cone} if $\mathrm{in}_C(I)$ is a prime ideal. We say that $X$ is \textit{well-poised} if every cone is a prime cone. 
\end{definition}

The term well-poised originates in the paper \cite{Ilten-Manon} by Ilten and the second author, where it was shown that every normal, rational affine variety $X$ equipped with an effective action by an algebraic torus $T$ of codimension $1$ has a well-poised embedding, see \cite[Theorem 1.2]{Ilten-Manon}. The embedding considered in \cite{Ilten-Manon} is known as the \emph{semi-canonical} embedding, and has properties comparable to that of the embedding of an affine toric variety by a minimal Hilbert basis of its coordinate ring.  The main construction of \cite{Ilten-Manon} allows for an explicit computation of any Newton-Okounkov cone or body of $X$ corresponding to a $T$-homogeneous full rank valuation. In this way, the results in \cite{Ilten-Manon} provide a generalization of the combinatorial equipment available for toric varieties.

More generally, the tropical variety of a well-poised variety encodes the data of many Newton-Okounkov cones \cite{Kaveh-Manon-NOK}, as well as piecewise-linear mutations between these cones \cite{Escobar-Harada}. A well-poised variety's tropical variety can also be shown to carry a continuous section of the tropicalization map from the Berkovich analytification (see \cite[\S 10]{Gubler-Rabinoff-Werner}). We will explore this construction in the case of arrangement varieties in future work.

Along these lines, we study the interaction of the well-poised property with the operation of taking a quotient by an algebraic torus. Unless stated otherwise we work over the field of complex numbers $\Cc$. Recall that a \emph{quasi-torus} is an algebrac group isomorphic to the product of a torus and a finite abelian group.  For an affine variety $X$ with coordinate ring $R_X$ equipped with an action by a quasi-torus $H$ with character $\alpha: H \to \Cc^\times$ we let $(R_X)_\alpha \subset R_X$ denote the character space of $\alpha$.  The following is proved in Section \ref{sec-quotients}.

\begin{theorem}\label{thm-main2}
Let $X$ be a well-poised normal variety equipped with an action by a quasi-torus $H$ with no non-constant $H$-invariant regular functions, and assume that $\dim_\Cc ((R_X)_\alpha) < \infty$ for all $\alpha$. Then the geometric invariant theory quotient $X \sslash_\alpha H$ has a well-poised embedding for any character $\alpha$.
\end{theorem}

As a first application of Theorem \ref{thm-main2}, we argue that the well-poised property is preserved after applying either a Veronese embedding to a well-poised projective variety $X$ or taking the Segre product of two well-poised projective varieties $X$ and $Y$. We also apply Theorem \ref{thm-main2} to certain Hassett spaces (see \cite{Alexeev}, \cite{Hassett}), and show that they have well-poised embeddings. Hassett spaces are generalizations of $\overline{M}_{0,n}$, so it is natural to ask if $\overline{M}_{0,n}$ has a well-poised embedding. In \cite[Theorem 1.2]{Maclagan-Gibney}, Gibney and Maclagan present $\overline{M}_{0,n}$ as a GIT quotient of an affine variety $X$ by a torus $T$, and they give equations up to saturation that cut out $X$. Thus, in order to show that this presentation of $\overline{M}_{0,n}$ is well-poised it would suffice to show it for $X$.

In Section \ref{sec-wpTvars}, we study the well-poised property in a setting inverse to Theorem \ref{thm-main2}. Instead of considering quotients of a well-poised variety by a quasi-torus, we ask when a normal $T$-variety $X$ is well-poised provided that its base $Y$ is well-poised. 

\begin{definition}
A \textit{$T$-variety} $X$ is a variety equipped with an action by the torus, $T \isom (\Cc^\times)^n$. The \textit{complexity} of $X$, $c(X) = \dim(X) - \dim(T)$, is the codimension of the largest $T$-orbit. The \textit{base} of $X$ is the normalized Chow quotient $X \sslash T$ which will often be denoted by $Y$.
\end{definition}

In general, $X$ can be recovered with a \textit{polyhedral divisor} $\D$ on $Y$ which is of the form
\[
    \sum_{i=0}^m \Delta_i \otimes D_i,
\]
where each $D_i$ is an effective Cartier divisor on $Y$ and the polyhedra $\{\Delta_i\}$ satisfy certain positivity conditions. Polyhedral divisors are the technical tools for $T$-varieties that play the same role rational polyhedral cones play for toric varieties. Section \ref{sec-TVarsPrelims} provides a brief introduction to the theory of polyhedral divisors. For a full treatment, we refer the reader to \cite{Altman-Hausen}. Our main technical result in Section \ref{sec-wpTvars} provides conditions for when a $T$-variety built on a well-poised base has a well-poised embedding. The embeddings considered are referred to as \textit{semi-canonical}, and they mirror the construction from \cite{Ilten-Manon} for rational complexity 1 $T$-varieties.

\begin{theorem}\label{thm-technical}
Let \(\varphi : Y \to \Pp^m\) be a projectively normal embedding, and let \(\mathfrak{D}\) be a polyhedral pullback along \(\varphi\) of a torus-invariant polyhedral divisor on \(\Pp^m\). Let \(I_Y\) be the ideal of \(Y^\circ = Y \cap T_{\Pp^m}\), let \(\mathcal{G}\) be a suitable generating set (see Section \ref{subsec-semi-canonical}), and let $J_\D$ be the vanishing ideal of $X(\D)$ under its semi-canonical embedding. Then, $\Trop(J_\D) \isom \Trop(I_Y) \times N_\Rr$, and a prime cone, \(C\), of \(\Trop(I_Y)\) lifts to a prime cone of \(\Trop(J_\D)\) whenever \(\deg (\mathrm{in}_w(g)) = \deg(g)\) for a fixed $w \in \mathrm{relint}(C)$ and all \(g \in \mathcal{G}\).
\end{theorem}

Theorem \ref{thm-technical} provides a road map for building well-poised $T$-varieties whose bases are well-poised. For example, it was shown in \cite{Speyer-Sturmfels} that $\mathrm{Gr}(2,n)$ is well-poised under its Pl{\"u}cker embedding. Moreover, the Pl{\"u}cker relations turn out to satisfy the conditions needed in Theorem \ref{thm-technical}; thus, any polyhedral pullback of a torus-invariant polyhedral divisor on $\Pp^{\binom{n}{2}-1}$ will produce a $T$-variety whose semi-canonical embedding is well-poised.

In Section \ref{genArrVar}, we apply Theorem \ref{thm-technical} to affine arrangement varieties. An arrangement variety $X$ is a $T$-variety whose base is $\Pp^c$. Note that rational complexity 1 $T$-varieties are arrangement varieties. The term arrangement is used, because the associated polyhedral divisor $\D$ on $\Pp^c$ can be viewed as a full rank hyperplane arrangement $\AA\subset \Pp^c$ where each hyperplane is decorated by a polyhedron. The hyperplane arrangement $\AA$ induces a well-poised, linear embedding of $\Pp^c$ in $\Pp^{m}$ where $m = |\AA| - 1$, and there is a $T_{\Pp^{m}}$-invariant polyhedral divisor whose polyhedral pullback to $\Pp^c$ is $\D$. It can then be checked that the circuits of the linear ideal cutting out the image of the base satisfy all the conditions in Theorem \ref{thm-technical}. 

\begin{theorem}\label{thm-main}
Let $X$ be an affine arrangement variety. Any semi-canonical embedding of $X$ is well-poised. 
\end{theorem}

Section \ref{genArrVar} ends with a description of the toric degenerations of a semi-canonically embedded, affine \textit{general} arrangement variety $X$ coming from its tropicalization. If $X = V(I)$ is its semi-canonical presentation, then as we shall see, $\Trop(I)$ is isomorphic to the Bergman fan of a uniform matroid crossed with a lineality space. This allows us to give explicit descriptions of the corresponding value semigroups (Theorem \ref{value-semigroup}) and Newton-Okounkov cones (Corollary \ref{cor-NOKcone}).

In Section \ref{sec-hypertoric}, we apply Theorems \ref{thm-main2} and \ref{thm-main} to hypertoric varieties which are hyperk{\"a}hler analogues of toric varieties and have been studied widely, e.g. \cite{Hausel-Sturmfels}, \cite{Proudfoot}, \cite{Kutler}, \cite{Bielawski-Dancer}, etc. Hypertoric varieties can be constructed as GIT quotients of certain subvarieties of the cotangent bundle $T^* \Aa^d$ (see \cite{Hausel-Sturmfels}). We argue that these subvarieties are in fact semi-canonically embedded arrangement varieties in $T^*\Aa^d \isom \Aa^{2d}$, so Theorems \ref{thm-main2} and \ref{thm-main} imply that these too have well-poised embeddings.

\subsection{Acknowledgements}

We thank Nathan Ilten for the suggestion to study the well-poised property for arrangement varieties. We also thank Max Kutler and Milena Wr\"obel for useful conversations. The second author was supported by Simons Collaboration Grant 587209 during this project. 

\section{Preliminaries}\label{sec-preliminaries}

\subsection{Valuations, Khovanskii bases, and prime cones}\label{prelim-valuations}

In this section, we introduce the constructions we will need from tropical geometry and the theory of valuations.  For background on tropical geometry, we suggest the book by Maclagan and Sturmfels \cite{Maclagan-Sturmfels}.  For background on Khovanskii bases and prime cones see \cite{Kaveh-Manon-NOK}, and for the notion of a well-poised ideal see \cite{Ilten-Manon}. Throughout this section, we will assume that $\mathbb{K}$ is a trivially valued, algebraically closed field.

Let $\Gamma$ be an abelian group equipped with a total group order $\prec$, and let $\bar{\Gamma} = \Gamma \cup \{\infty\}$.  Recall that a valuation $\mathfrak{v}: A \to \bar{\Gamma}$ over $\Kk$ is a function satisfying the following conditions for any $f,g \in A$ and any $C \in \mathbb{K}\setminus\{0\}$.
\begin{enumerate}
    \item $\mathfrak{v}(fg) = \mathfrak{v}(f) + \mathfrak{v}(g)$
    \item $\mathfrak{v}(f + g) \geq \mathrm{min}\{\mathfrak{v}(f), \mathfrak{v}(g)\}$
    \item $\mathfrak{v}(0) = \infty$
    \item $\mathfrak{v}(Cf) = \mathfrak{v}(f)$  
\end{enumerate}
In this paper we take $\Gamma = \Qq^r$ or $\Zz^r$ equipped with the lexicographic order. If $A$ is graded by an abelian group $M$, a valuation is said to be homogeneous with respect to $M$ if the value of a general element of $A$ is always obtained on one of its homogeneous components. 

For any $q \in \Gamma$ there is a $\Kk$-vector space 
\[
    F_{\succeq q}(\mathfrak{v}) = \{a \in A \mid \mathfrak{v}(a) \succeq q\}.
\]  

The subspace $F_{\succ q}(\mathfrak{v}) \subseteq F_{\succeq q}(\mathfrak{v})$ is defined similarly. 
It is easy to check that $F_{\succeq q}(\mathfrak{v})F_{\succeq q'}(\mathfrak{v}) \subseteq F_{\succeq q + q'}(\mathfrak{v})$ and if $q \prec q'$ then $F_{\succeq q}(\mathfrak{v}) \supseteq F_{\succeq q'}(\mathfrak{v})$; in this way the spaces $F_{\succeq q}(\mathfrak{v})$ form a $\Kk$-algebra filtration of $A$.  We let
\[
    \mathrm{gr}_\mathfrak{v}(A) = \bigoplus_{q \in \Gamma}F_{\succeq q}(\mathfrak{v})/F_{\succ q}(\mathfrak{v})
\]\noindent
denote the \emph{associated graded} algebra of $\mathfrak{v}$. It is straightforward to check that the properties of valuations ensure that $\mathrm{gr}_\mathfrak{v}(A)$ is $\Kk$-domain. 

We let $S(A, \mathfrak{v}) \subseteq \Gamma$ denote the set of finite values of $\mathfrak{v}$.  This set contains $0$ and is closed under the group operation in $\Gamma$, so it is referred to as the \emph{value semigroup} of the valuation $\mathfrak{v}$.  The rank $\mathrm{rank}(\mathfrak{v})$ is defined to be the rank of the subgroup of $\Gamma$ generated by $S(A, \mathfrak{v})$. We say that $\mathfrak{v}$ has \emph{full rank} if $\mathrm{rank}(\mathfrak{v})$ is equal to the Krull dimension of $A$ as a $\Kk$-algebra.  Under our assumptions if $\mathfrak{v}$ is full rank then $\mathrm{gr}_\mathfrak{v}(A)$ is isomorphic to the semigroup algebra $\Kk[S(A, \mathfrak{v})]$, see \cite[ Proposition 2.4]{Kaveh-Manon-NOK}.  The closure of the convex hull $\mathrm{conv}(S(A,\mathfrak{v}))$ is denoted $P(A, \mathfrak{v})$; it is known as the Newton-Okounkov cone of $\mathfrak{v}$.

Moreover, if $A$ is positively graded, one can define the Newton-Okounkov body $\Delta(A, \mathfrak{v})$ to be the closure of $\mathrm{conv}\{ \frac{\mathfrak{v}(f)}{\deg(f)} \mid \mathfrak{v}(f) < \infty\}$ (see \cite{Lazarsfeld-Mustata}, \cite{Kaveh-Khovanskii}, \cite{Kaveh-Manon-NOK}).  It is known that the Newton-Okounkov body $\Delta(A, \mathfrak{v})$ carries algebraic information about $A$ and geometric information about the $\Kk$-scheme $\textup{Proj}(A)$.  

With respect to the valuation $\mathfrak{v}$, a \emph{Khovanskii basis} (\cite[Definition 2.5]{Kaveh-Manon-NOK}) is a $\Kk$-algebra generating set $\mathcal{B} \subset A$ with the property that the equivalence classes $\overline{\mathcal{B}} \subset \mathrm{gr}_\mathfrak{v}(A)$ also form a $\Kk$-algebra generating set. For example, if $\mathfrak{v}$ is full rank, then any generating set whose equivalence classes in $\mathrm{gr}_\mathfrak{v}(A)$ form a system of semigroup generators of $S(A, \mathfrak{v})$ is a Khovanskii basis. If a $\mathcal{B} \subset A$ is a finite Khovanskii basis of a full rank valuation $\mathfrak{v}$, then $\Delta(A, \mathfrak{v}) = \mathrm{conv}\{\frac{\mathfrak{v}(b)}{\deg(b)} \mid b \in \mathcal{B}\}$.  Moreover, a result of Anderson \cite{Anderson} states that in this case, $\textup{Proj}(A)$ can be degenerated to a projective toric variety whose normalization is the toric variety associated to the polytope $\Delta(A, \mathfrak{v})$.

Let $\phi: \Kk[{\bf x}] \to A$ be the map from a polynomial ring on $n=|\mathcal{B}|$ variables determined by $\mathcal{B} \subset A$, and let $I = \ker(\phi)$.  Theorem 2 of \cite{Kaveh-Manon-NOK} connects valuations with finite Khovanskii basis $\mathcal{B}$ to the structure of the tropical variety $\Trop(I)$. First, recall that for any point ${\bf u} \in \Rr^n$ and polynomial $f = \sum C_\alpha{\bf x}^\alpha \in \Kk[{\bf x}]$ there is an associated \emph{initial form}, $\mathrm{in}_{\bf u}(f)$, obtained as the sum of those monomial terms $C_\alpha{\bf x}^\alpha$ for which $\langle {\bf u}, \alpha \rangle$ is minimized.  For an ideal $I \subset \Kk[{\bf x}]$, the \emph{initial ideal}, $\mathrm{in}_{\bf u}(I) \subset \Kk[{\bf x}]$, is the ideal generated by the initial forms of the elements of $I$. For both of these notions see \cite{Maclagan-Sturmfels} or \cite{GBCP}.  A key result of the theory of Gr\"obner bases is that there are only finitely many initial ideals, and the equivalence classes $C_{\bf u} = \{ {\bf w} \mid \mathrm{in}_{\bf w}(I) = \mathrm{in}_{\bf u}(I)\}$ form the relatively open faces of a polyhedral fan $\Sigma(I)$ in $\Rr^n$ called the Gr\"obner fan.  If $I$ is homogeneous, the fan $\Sigma(I)$ is complete. In this case, the tropical variety $\Trop(I) \subset \Rr^n$ is realized as the set of those ${\bf u}$ such that $\mathrm{in}_{\bf u}(I)$ contains no monomials, and it is known that $\Trop(I)$ is the support of a subfan of $\Sigma(I_h)$, where $I_h$ is the homogenization of $I$.  The following is \cite[Theorem 4]{Kaveh-Manon-NOK}.

\begin{theorem}
Let $\phi: \Kk[{\bf x}] \to A$ be a presentation with ideal $I$ and let $\mathcal{B} = \phi({\bf x})$.  Let $C \subset \Trop(I)$ be a relatively open cone
of full dimension $(= d)$ such that $\mathrm{in}_{\bf u}(I)$ is prime for all ${\bf u} \in C$.  Then for every linearly independent subset $\{ {\bf u}_1, \ldots, {\bf u}_d\} \subset C$ there is a full rank valuation $\mathfrak{v}: A\setminus\{0\} \to \Qq^d$ with Khovanskii basis $\mathcal{B}$, such that for any ${\bf u} \in C$ we have $\mathrm{gr}_{\mathfrak{v}}(A) \cong \Kk[{\bf x}]/\mathrm{in}_{\bf u}(I)$.
\end{theorem}

The cone $C$ is said to be a \emph{prime cone} because the initial ideal $\mathrm{in}_{\bf u}(I)$, which is constant over ${\bf u} \in C$, is a prime ideal. 

\begin{definition}
Let $\phi: \Kk[{\bf x}] \to A$ be a presentation with ideal $I$. We say that $\Trop(I)$ is \textit{well-poised} if every relatively open cone $C$ of $\Trop(I)$ is a prime cone.
\end{definition}

To build the valuation $\mathfrak{v}$, one forms the matrix $M$ with rows equal to the ${\bf u}_i$.  The matrix $M$ defines a rank $d$ valuation on the polynomial ring $\Kk[{\bf x}]$ by sending a monomial ${\bf x}^{\alpha}$ to the vector $M\alpha \in \Qq^r$. A polynomial $f = \sum C_\alpha{\bf x}^\alpha \in \Kk[{\bf x}]$ is sent to $\tilde{\mathfrak{v}}_M(f) = \mathrm{min}\{ M\alpha \mid C_\alpha \neq 0\}$, where $\mathrm{min}$ is taken with respect to the lexicographic ordering on $\Qq^r$.  The valuation $\mathfrak{v}_M: A\setminus \{0\} \to \Qq^r$ is then defined to be the \emph{pushforward} of $\tilde{\mathfrak{v}}_M$:

\[\mathfrak{v}_M(f) = \mathrm{max}\{ \tilde{\mathfrak{v}}(p) \mid \phi(p) = f \}.\]

\noindent
Once again, $\mathrm{max}$ is taken with respect to the lexicographic ordering on $\Qq^r$. In this case, the value semigroup $S(A, \mathfrak{v}_M)$ can be computed as the $\Zz_{\geq 0}$-combinations of the columns of $M$.

\subsection{$T$-Varieties and Polyhedral Divisors}\label{sec-TVarsPrelims}
Let \(T\) be an algebraic torus with character and co-character lattices \(M\) and \(N\), respectively. Recall that a normal affine complexity \(c\) \(T\)-variety is a normal affine variety, \(X\), equipped with an effective torus action, \(T\times X \to X\), so that \(c = \dim(X) - \dim(T)\). Such varieties can be built out of a base variety \(Y\) with \(\dim(Y)= c\) and some combinatorial data. The base variety can be taken to be the normalized Chow quotient of $X$ by $T$. Together this information is known as a {\it polyhedral divisor} on \(Y\) (see \cite{Altman-Hausen}).

Here we will give a brief introduction to the correspondence between affine \(T\)-varieties and polyhedral divisors. Fix a normal semi-projective base, \(Y\), and a rational pointed polyhedral cone \(\sigma \subseteq N_\Qq := N\otimes \Qq\). A polyhedron is a \(\sigma\)-\textit{polyhedron} if it can be written as \(\Pi + \sigma\) where \(\Pi\subseteq N_\Qq\) is a rational polytope and \(+\) denotes Minkowski sum. Then a polyhedral divisor on \(Y\) is a formal sum
\[
    \mathfrak{D} = \sum_i \Delta_i \otimes D_i
\]
where each \(D_i\) is a distinct effective Cartier divisor on \(Y\) and each \(\Delta_i\) is a rational \(\sigma\)-polyhedron in \(N_\Qq\), and all but finitely many $\Delta_i$ are equal to $\sigma$. Note that we do allow the empty set as a $\sigma$-polyhedron.

A non-empty $\sigma$-polyhedron $\Delta$ induces a piece-wise linear map defined on the dual cone $\sigma^\vee$ given by 
\[
    \Delta(u) = \mathrm{min}_{v\in\Delta}\langle u,v \rangle.
\]
Note that for any $u \in \sigma^\vee$, the minimum is always achieved at a vertex of $\Delta$. Now, any polyhedral divisor \(\mathfrak{D}\) gives a piece-wise linear map 
\begin{align*}
    \mathfrak{D} : \sigma^\vee &\to \mathrm{Div}_{\Qq \cup \infty} (Y) \\
    u &\mapsto \sum_{i} \Delta_i(u) \cdot D_i
\end{align*}
where \(\mathrm{Div}_{\Qq\cup\infty}(Y)\) is the semigroup of finite formal sums of Weil divisors on \(Y\) whose coefficients are taken in \(\Qq\cup \infty\). By convention, \(\Delta_i(u) = \infty\) if and only if \(\Delta_i = \emptyset\). If for every \(u \in \sigma^\vee \cap M\), \(\mathfrak{D}(u)\) is semi-ample, and \(\mathfrak{D}(u)\) is big whenever \(u \in \text{relint}(\sigma^\vee) \cap M\), we can construct a normal affine \(T\)-variety
\[
    X(\mathfrak{D}) = \Spec \bigoplus_{u\in\sigma^\vee \cap M} H^0(Y, \OO(\lfloor \mathfrak{D}(u) \rfloor)) \cdot \chi^u.
\]
In this case, we say \(\mathfrak{D}\) is a {\it proper} polyhedral divisor (see \cite{Altman-Hausen}). If one allows for the $D_i$ to repeat, $\D$ might not even be a polyhedral divisor; however, one can still build an $M$-graded algebra in the same manner as above. The key difference is that the resulting $T$-variety may fail to be normal.

In this paper, we are interested in the case when the base, \(Y \subseteq \Pp^m\), is projectively normal, and each \(D_i\) is the pullback of a coordinate hyperplane $V(x_i) \subset \Pp^m$. In this case, the positivity conditions amount to checking \(\sum_{i} \Delta_i \subsetneq \sigma\). Indeed, semi-amplitude, bigness, and being Cartier are all properties which are preserved under pullbacks, so it suffices to check that the polyhedral divisor has these properties on $\Pp^m$. This is the case as long as each $\D(u)$ is equivalent to an effective divisor which happens if and only if the Minkowski sum, $\sum_i \Delta_i$, is strictly contained in $\sigma$.

%%%%%%%%%%%%%%%%%%%%%%%%%%
%       QOUTIENTS OF WP T-VARS
%%%%%%%%%%%%%%%%%%%%%%%%%%
%Dolgachev and Mumford references

\section{GIT Quotients of Well-poised Varieties}\label{sec-quotients}
In this section, we will study how the well-poised condition interacts with geometric invariant theory (GIT) quotients. In particular, given a well-poised embedding of a $T$-variety $X$ and a character $\alpha$, there is a well-poised embedding of the GIT quotient $X \sslash_\alpha T$. For the basics of geometric invariant theory, see \cite{Dolgachev}, \cite{Mumford}.  

In Corollaries \ref{cor-veronese} and \ref{cor-segre}, we apply Theorem \ref{thm-quotients of wellpoised} to Veronese and Segre embeddings of well-poised varieties.  In Corollary \ref{cor-hassett} we show that Hassett spaces with small weights \cite{Alexeev} have well-poised embeddings. Results from this section will be used later in Section \ref{sec-hypertoric} to study hypertoric varieties.

\begin{theorem}\label{thm-quotients of wellpoised}
Let $R = \Cc[x_1,\dotsc,x_n]/I$ be a well-poised presentation of the coordinate ring of a normal variety $X$. Moreover, assume that $R$ is graded by a finitely generated abelian group $M$, $\dim_\Cc(R_\alpha) < \infty$ for all $\alpha \in M$, and $R_0 = \Cc$. It follows that $X$ has an effective action by the quasi-torus $T = \Hom(M,\Cc^\times)$. Then, there is a well-poised embedding of the GIT quotient $X \sslash_\beta T$ for any choice of character $\beta \in M$.
\end{theorem}

\begin{proof}
Let $R = \bigoplus_{\alpha \in M} R_\alpha$ be a finitely generated graded algebra with $\dim_\Cc (R_\alpha) < \infty$. Let $\mathfrak{v}: R \to \Gamma$ be a homogeneous valuation with finite Khovanskii basis, and let $S_\beta = \bigoplus_{n \geq 0} R_{n\beta}$ for $\beta \in M$, with corresponding inclusion $\iota: S_\beta \to R$. The algebra $S_\beta$ is a ring of invariants for the action of a quasi-torus $T$, so $S_\beta$ must be finitely generated.  Moreover, the same is true for the subring $\mathrm{gr}_\mathfrak{v}(S_\beta) \subset \mathrm{gr}_\mathfrak{v}(R)$.  After passing to a sufficiently high Veronese subring, we may assume that both $\mathrm{gr}_\mathfrak{v}(S_\beta)$ and $S_\beta$ are generated by their first graded component.  Observe that if $x_1, \ldots, x_n \in R$ is a Khovanskii basis for $\mathfrak{v}$, and $\mathrm{gr}_\mathfrak{v}(S_\beta)$ is generated by its first homogeneous component, then the monomials $y_1, \ldots, y_d \in R_\beta$ of degree $\beta$ are a Khovanskii basis for $S_\beta$.

Now let $R = \Cc[x_1, \ldots, x_n]/I$ be a well-poised presentation of $R$, then the images of the $x_1, \ldots, x_n$ are a Khovanskii basis for the weight valuation $\mathfrak{v}_u$ for any $u \in \Trop(I)$.  Fix a $\beta \in M$ as above.  Observe that the associated graded algebra $\mathrm{gr}_{\mathfrak{v}_u}(R) \cong \Cc[x_1, \ldots, x_n]/\mathrm{in}_u(I)$ depends only on the relative interior of the face $\sigma$ that $u$ lies on.  It follows that $\mathrm{gr}_{\mathfrak{v}_u}(S_\beta)$ also depends only on $\sigma$.  Now take $n_\sigma$ to be sufficiently large so that $\mathrm{gr}_{\mathfrak{v}_u}(S_{n_\sigma \beta})$ is generated by its first homogeneous component.  From now on we replace $\beta$ with $N\beta$ where $N = \prod_{\sigma \in \Trop(I)} n_\sigma$.  It follows that $\mathrm{gr}_{\mathfrak{v}_u}(S_\beta)$ is generated by its first homogeneous component for all $u \in \Trop(I)$.  Moreover, by above the monomials $y_1, \ldots, y_d$ of degree $\beta$ form a Khovanskii basis for the restriction of $\mathfrak{v}_u$ to $S_\beta$ for any $u \in \Trop(I)$. 

Now let $J$ be the ideal of forms which vanish on $y_1, \ldots, y_d \in S_\beta$, and let $\mathrm{Trop}(\phi^*): \Trop(I) \to \Trop(J)$ be the induced linear map on tropical varieties obtained by tropicalizing the monomial map $\phi^*$ induced by the $y_j$'s. The previous paragraph shows that any $w \in \text{Im}(\mathrm{Trop}(\phi^*))$ has $\mathrm{in}_w(J)$ prime.  But, $\phi^*$ is a monomial map, and is therefore a surjective map on tropical varieties by \cite[Lemma 3.2.13]{Maclagan-Sturmfels}.  It follows that the presentation of $S_\beta$ by the monomials $y_1, \ldots, y_d$ of degree $\beta$ is well-poised.
\end{proof}

Some of the most fundamental operations in algebraic geometry can be realized as a GIT quotient by a torus. Theorem \ref{thm-quotients of wellpoised} then implies that these operations preserve the well-poised property. In particular, Veronese embeddings and Segre products preserve the well-poised property.

\begin{corollary}\label{cor-veronese}
Let $X = V(I) \subseteq \Pp^n$ be a projectively normal, well-poised variety and consider the $d^\text{th}$ Veronese of $X$ in $\Pp^{\binom{n+d}{d} - 1}$. This is also a well-poised embedding.
\end{corollary}

\begin{proof}
Let $R = \Cc[x_0,\dotsc, x_n]/I = \bigoplus_{n\geq 0} R_n$. As this embedding is well-poised, $\overline{x_0},\dotsc,\overline{x_n}$ form a Khovanskii basis for any valuation $\mathfrak{v}_w$ where $w \in \Trop(I)$. The $d^\text{th}$ Veronese of $X$ is $\widehat{X}\sslash_d \Cc^\times = \mathrm{Proj}(S)$ where $S \subset R$ given by
\[
    \bigoplus_{n\geq 0} R_{nd}
\]
Let $J \subseteq \Cc[y_\alpha ~|~ \alpha \in \Zz_{\geq 0}^{n+1}, \sum_i \alpha_i = d]$ be the ideal of forms which vanish on the monomials spanning $R_d$, so $V(J) \subseteq \Pp^{\binom{n+d}{d} -1}$ is the image of the $d^\text{th}$ Veronese embedding of $X$. 

Note that the affine cone $\widehat{X}$ is normal as $X$ is projectively normal. Following the proof of Theorem \ref{thm-quotients of wellpoised}, $J$ will be well-poised as long as the monomials in the first homogeneous component of $S$ form a Khovanskii basis, but this is the case as $\overline{x_0},\dotsc,\overline{x_n}$ form a Khovanskii basis for $R$.
\end{proof}

\begin{corollary}\label{cor-segre}
Let $X = V(I) \subseteq \Pp^n$ and $Y = V(J) \subseteq \Pp^m$ be two normal well-poised varieties. Then the Segre embedding $X \times Y \subseteq \Pp^{(n+1)(m+1) - 1}$ is well-poised.
\end{corollary}

\begin{proof}
Let $I \subseteq R = \Cc[x_0,\dotsc,x_n]$ and $J \subseteq S = \Cc[y_0,\dotsc,y_m]$. We claim that $I + J \subseteq T = \Cc[x_0,\dotsc,x_n,y_0,\dotsc,y_m]$ is well-poised. Note that $I+J \subseteq T$ is the ideal for $\widehat{X} \times \widehat{Y} \subseteq \Aa^{m+n+2}$ and that $\widehat{X}$ and $\widehat{Y}$ are both normal as $X$ and $Y$ are projectively normal. The following properties are easily checked.
\begin{enumerate}
    \item $\Trop(I+J)$ is  $\Trop(I) \times \Trop(J) \subseteq \Rr^{n+1} \times \Rr^{m+1}$.
    \item For $w = (w_1,w_2) \in \Rr^{n+1} \times \Rr^{m+1}$, $\mathrm{in}_w (I+J) = \mathrm{in}_{w_1}(I) + \mathrm{in}_{w_2}(J)$.
    \item If $Z_1$ and $Z_2$ are both irreducible/normal affine varieties, then their product is irreducible/normal as well.
\end{enumerate}
It follows that $T/(I+J)$ is a well-poised presentation for $\widehat{X} \times \widehat{Y}$ and that the images of the variables form a Khovanskii basis. The polynomial ring $T$ is equipped with the standard bigrading
\[
\deg(x_i) = (1,0) \text{ and }\deg(y_j) = (0,1) 
\]
and $I+J$ is clearly bigraded. The product $X \times Y$ can be realized as the GIT quotient $\widehat{X} \times \widehat{Y} \sslash_{(1,1)} (\Cc^\times)^2 = \mathrm{Proj} (T')$ where $T'$ is the following subring of $T$.
\[
    \bigoplus_{n\geq 0} (T/(I+J))_{(n,n)}
\]
Let $K \subseteq \Cc[z_{ij}]$ be the ideal of forms which vanish on the variables spanning the first homogeneous component of $T'$, and note that $V(K) \subseteq \Pp^{(n+1)(m+1) - 1}$ is the Segre product $X\times Y$.

Following the proof of Theorem \ref{thm-quotients of wellpoised}, one sees that $K$ is well-poised as long as the monomials $\overline{x_iy_j}$ form a Khovanskii basis for $T'$. Since the variables $\overline{x_0}, \dotsc, \overline{x_n},\overline{y_0}, \dotsc, \overline{y_m}$ form a Khovanskii basis for $T$, it follows that the monomials $\overline{x_i y_j}$ form a Khovanskii basis for $T'$. 
\end{proof}

Theorem \ref{thm-quotients of wellpoised} can also be applied to Hassett spaces with small weights. As discussed below, these spaces can be viewed as GIT quotients of $\mathrm{Gr}(2,n)$ by a torus (see \cite{Alexeev}), so Theorem \ref{thm-quotients of wellpoised} applies as $\mathrm{Gr}(2,n)$ is well-poised (see \cite{Speyer-Sturmfels}). We recall the definition of the Hassett space, $\overline{M}_{0,\beta}$ for an $n$-tuple $\beta$, and give all the relevant definitions and theorems needed to state Theorem \ref{thm-Alexeev}. We refer the reader to \cite{Alexeev} or \cite{Hassett} for a full treatment.

\begin{definition}
Fix $\beta \in ((0,1] \cap \Qq)^n$. A $\beta$-weighted stable curve of genus 0 is a nodal curve $C = \cup \Pp^1$ whose dual graph is a tree with $n$ marked points $P_1,\dotsc,P_n$ satisfying the following two conditions:
\begin{enumerate}
    \item None of the $P_i$'s can be nodes. Moreover, if the points $\{P_i ~|~ i \in I\}$ coincide, then $\sum_{i\in I} \beta_i \leq 1$.
    \item On each irreducible component $C'$ of $C$, one must have $\#\text{nodes} + \sum_{P_i \in C'} \beta_i > 2$, i.e. $K_C + \sum_{i=1}^n \beta_i P_i$ must be an ample divisor.
\end{enumerate}
The fine moduli space of flat families over such curves is known as the Hassett space $\overline{M}_{0,\beta}$.
\end{definition}

Hassett spaces are natural generalizations of $\overline{M}_{0,n}$, and in fact, $\overline{M}_{0,n}$ can be recovered by setting $\beta = (1,\dotsc,1)$. Of course, as $\beta$ varies, the Hassett spaces change as well. To this end, we recall the weight domain $\mathcal{D}(n)$ and its subdivision which can be found in a paper of Alexeev, \cite[Definition 1.2]{Alexeev}. By \cite[Theorem 1.4]{Alexeev}, we see that $\overline{M}_{0,\beta} = \overline{M}_{0,\beta'}$ whenever $\beta$ and $\beta'$ lie in the same chamber.

\begin{definition}
The \textit{weight domain} $\mathcal{D}(n)$ is 
\[
\{\beta \in \Qq^n ~|~ 0 < \beta_i \leq 1, \; \sum \beta_i > 2\}.
\]
The weight domain is subdivided into locally closed chambers $\mathrm{Ch}(\beta)$ by the hyperplanes $\sum_{i \in I} x_i = 1$ for all subsets $I \subseteq \{1,\dotsc,n\}$.
\end{definition}

\begin{theorem}\label{thm-Alexeev}
\cite[Theorem 1.5 (Moduli for small weights)]{Alexeev} Let $\alpha \in \Qq^n$ be a weight with $\sum \alpha_i = 2$ lying on the boundary of a unique chamber $\mathrm{Ch}(\beta)$. Then 
\[
    \overline{M}_{0,\beta} = \mathrm{Gr}(2,n) \sslash_\alpha (\Cc^\times)^{n-1}.
\]
\end{theorem}

\begin{corollary}\label{cor-hassett}
In the situation of the preceding theorem, $\overline{M}_{0,\beta}$ has a well-poised embedding.
\end{corollary}

\begin{proof}
Since $\mathrm{Gr}(2,n)$ is well-poised, and $\overline{M}_{0,\beta}$ is a GIT quotient of the Grassmannian by Theorem \ref{thm-Alexeev}, we can apply Theorem \ref{thm-quotients of wellpoised}.
\end{proof}

In \cite[Theorem 1.2]{Maclagan-Gibney}, Maclagan and Gibney find an ideal $I_{\overline{M}_{0,n}}$ that cuts out $\overline{M}_{0,n}$ in the Cox ring, $S$, of some toric variety $X_\Delta$. The generators of $I_{\overline{M}_{0,n}}$ (up to saturation by the product of the variables of $S$) arise from the Pl{\"u}cker relations coming from $\mathrm{Gr}(2,n)$. It follows that $\overline{M}_{0,n}$ is the GIT quotient, $V(I_{\overline{M}_{0,n}}) \sslash_\alpha H$ where $H = \mathrm{Hom}(\mathrm{Cl}(X_\Delta), \Cc^\times)$ and $\alpha \in \mathrm{Cl}(X_\Delta)$ is some suitable character. With all this in mind, we ask the following question.

\begin{question}
Is Maclagan and Gibney's embedding of $\overline{M}_{0,n}$ or some Veronese of it described above well-poised? Or equivalently, in view of Theorem \ref{thm-quotients of wellpoised}, is $I_{\overline{M}_{0,n}}$ a well-poised ideal?
\end{question}

%%%%%%%%%%%%%%%%%%%%%%%%%%
%  Quot WP => T-variety WP
%%%%%%%%%%%%%%%%%%%%%%%%%%

\section{Well-poised $T$-varieties}\label{sec-wpTvars}

In this section, we study the tropicalizations of affine $T$-varieties whose Chow quotients are projectively normal. Given such a variety $X$, we provide an embedding of $X \subseteq \Aa^n$ which depends on its polyhedral divisor $\D$ and an embedding of its base $Y \subseteq \Pp^m$. We show that under this embedding $\Trop(X) \isom \Trop(Y) \times N_\Rr$, and we prove Theorem \ref{thm-technical}.

%%%%%%%%%%%%%%%%%%%%%%%%%%
% Semi-canonical Embeddings
%%%%%%%%%%%%%%%%%%%%%%%%%%

\subsection{Semi-canonical embeddings}\label{subsec-semi-canonical}

In \cite{Ilten-Manon}, the authors consider \textit{semi-canonical embeddings} of affine rational complexity 1 $T$-varieties. Our first goal is to generalize these embeddings to $T$-varieties whose base is not necessarily $\Pp^1$. 

An affine rational complexity 1 $T$-variety, $X$, is built from a polyhedral divisor on $\Pp^1$ of the form $\D = \Delta_0 \otimes P_0 + \dotsc + \Delta_m \otimes P_m$ where each $P_i$ is a distinct point on $\Pp^1$ and the polyhedra have a common tail-cone $\sigma$ so that $\sum_i \Delta_i \subsetneq \sigma$. There is a linear embedding $\eta : \Pp^1 \to \Pp^m$ so that the pullback of each torus-invariant divisor, $H_i$, on $\Pp^m$ is $P_i$. The torus-invariant polyhedral divisor $\D_\text{toric} = \sum_{i=0}^m \Delta_i \otimes H_i$ on $\Pp^m$ produces a toric variety $X(\D_\text{toric})$. There is an embedding of $X$ into $X(\D_\text{toric})$ induced by the maps given by restricting sections to $\eta(\Pp^1)$:
\[
    H^0(\Pp^m, \OO(\D_\text{toric}(u))) \to H^0(\Pp^1, \OO(\D(u))).
\]
This is the semi-canonical embedding of a rational complexity 1 $T$-variety.

Moving away from the rational complexity 1 case, assume we have an affine $T$-variety, $X$, with polyhedral divisor $\D = \sum_{i=0}^m \Delta_i \otimes D_i$ on $Y$. Moreover, assume that we have a projectively normal embedding $\eta : Y \to \Pp^m$ so that $\eta^*(H_i) = D_i$ for each $i$ and that $\sum_i \Delta_i \subsetneq \sigma$. Let $\D_\text{toric} = \sum_{i=0}^m \Delta_i \otimes H_i$. The first condition guarantees that the restriction maps
\[
    H^0(\Pp^m, \OO(\D_\text{toric}(u))) \to H^0(Y, \OO(\D(u)))
\]
are surjective for each $u \in \sigma^\vee \cap M$. The second condition implies that $\D$ is a proper polyhedral divisor as semi-amplitude, bigness, and the Cartier property are all preserved under pullbacks. In the language of \cite{Altman-Hausen}, we say that $\D$ is a \textit{polyhedral pullback} of $\D_\text{toric}$.

These restriction maps piece together to give a presentation of $R(\D)$ by the ring below.
\[
R(\D_\text{toric}) := \bigoplus_{u \in \sigma^\vee \cap M} H^0(\Pp^m, \OO(\D_\text{toric}(u))) \cdot \chi^u
\]
The next proposition shows that this ring is a saturated semigroup algebra, and we give its corresponding rational polyhedral cone.

\begin{proposition} \label{Ambient Toric Variety Prop}
Let $\D = \sum_{i=0}^m \Delta_i \otimes D_i$ be a polyhedral divisor on $Y$ as above. Define $\delta \subseteq \Rr^m \times N_\Rr$ to be the cone obtained by taking the positive hull of
\[
    (\{{\bf 0}\}\times\sigma) \cup \bigcup_{i=0}^m \{e_i\} \times \Delta_i
\]
where $e_0 = - \sum_{i=1} e_i$ and $\{e_1,\dotsc,e_m\}$ is the standard basis for $\Rr^m$. Then $R(\D_\text{toric})$ and $\Cc[\delta^\vee \cap (\Zz^m \times M)]$ are isomorphic as semigroup algebras. In particular, $X(\D_\text{toric})$ is a normal affine toric variety.
\end{proposition}

\begin{proof}
Consider the direct sum decomposition of $R(\D_\text{toric}) = \bigoplus_{u \in \sigma^\vee \cap M} H^0(\Pp^m, \OO(\D_\text{toric}(u))) \cdot \chi^u$. Note that for every $u \in \sigma^\vee \cap M$, $\D_\text{toric}(u)$ is $T_{\Pp^m}$-invariant, and there is an associated polyhedron whose lattice points correspond to a monomial basis of $\D_\text{toric}(u)$. We will denote this polyhedron by $P_u$. The direct decomposition can be refined to get 
\[
    R(\D_\text{toric}) = \bigoplus_{(v,u) \in S} \mathrm{span}_\Cc \{t^v\cdot \chi^u\}
\]
where $S = \{(v,u) \in \Zz^m \times M ~|~ u \in \sigma^\vee \text{ and }v \in P_u\}$. This shows that $X(\D_\text{toric})$ is an affine toric variety, so it remains to show that $S = \delta^\vee \cap (\Zz^m \times M)$.

If $\D_\text{toric}(u) = \sum_{i=0}^m \lfloor \Delta_i(u)\rfloor H_i$, recall that $P_u$ is defined by the inequalities
\[
    \langle v, e_i \rangle \geq -\lfloor\Delta_i(u)\rfloor
\]
for all $i = 0,\dotsc, m$. The ray generators of $\delta$ come in one of two types: they are either a multiple of $(e_i, w_{ij})$ where $w_{ij}$ is a vertex of $\Delta_i$ or they are of the form $(0, \rho)$ where $\rho$ is a ray generator of $\sigma$. It follows that $(v,u) \in \delta^\vee$ if and ond only if $\langle (e_i, w_{ij}), (v,u) \rangle \geq 0$ and $\langle (0,\rho), (v,u) \rangle \geq 0$ for all possible $\rho$ and $(e_i, w_{ij})$.

\begin{align*}
(v,u) \in \delta^\vee \cap (\Zz^m \times M) \iff& \langle (e_i, w_{ij}), (v,u) \rangle \geq 0 \text{ and } \langle (0,\rho), (v,u) \rangle \geq 0 \\
				\iff& v_i \geq - \lfloor\langle w_{ij}, u\rangle\rfloor  \text{ for all }i,j \text{ and } u \in \sigma^\vee \cap M\\
				\iff& v_i \geq -\mathrm{min}_j\{ \lfloor\langle w_{ij},u\rangle\rfloor\} \text{ for all }i \text{ and } u \in \sigma^\vee \cap M\\
				\iff& v_i \geq - \lfloor\Delta_i(u)\rfloor \text{ for all }i \text{ and } u \in \sigma^\vee \cap M\\
				\iff& v \in P_u \text{ and }u \in \sigma^\vee \cap M\\
				\iff& (v,u) \in S
\end{align*}
Thus, $X(\D_\text{toric})$ is an affine normal toric variety. 
\end{proof}

\noindent Below is the main definition of this section.

\begin{definition}\label{def-semicanonical}
Let $\eta : Y \to \Pp^m$ be a projectively normal embedding, and let $\D$ be a polyhedral pullback along $\eta$ of a torus-invariant polyhedral divisor $\D_\text{toric}$ on $\Pp^m$. Let $\delta$ be as in Proposition \ref{Ambient Toric Variety Prop}, and let $\mathcal{H}$ be a Hilbert basis for the dual cone. Then the embedding, $X(\D) \subseteq \Aa^{\#\mathcal{H}}$ induced by the following composition of surjective maps is said to be \textit{semi-canonical}.
\[
    \Cc[x_{(h,h')} ~|~ (h,h') \in \mathcal{H}] \to R(\D_\text{toric}) \to R(\D)
\]
The first map is given by $x_{(h,h')} \mapsto t^h \chi^{h'}$, and the second is given by restricting sections. The images of $x_{(h,h')}$ in $R(\D)$ are the \textit{semi-canonical generators of} $R(\D)$.
\end{definition}

The kernel of $\Cc[x_h ~|~ h \in \mathcal{H}] \to R(\D)$ will be denoted by $J_\D$ and the kernel of $R(\D_\text{toric}) \to R(\D)$ by $I_\D$. The remainder of this subsection is devoted to describing a generating set for the ideal $I_\D$. The description depends on a choice of generators of the vanishing ideal for the very affine variety $Y^\circ := Y \cap T_{\Pp^m}$ and the polyhedral coefficients of $\D$. The vanishing ideal for $Y^\circ$ will be denoted by $I_Y$. Here is a first pass at a set of generators for $I_\D$.

\begin{lemma}\label{first gen set}
Let $I_\D$ and $I_Y$ be as above. Then $I_\D = \langle gt^v \cdot \chi^u ~|~ g \in I_Y \text{ and } gt^v \cdot \chi^u \in R(\D_\text{toric})\rangle$. Geometrically, this means that $X(\D)$ is the Zariski closure of $Y^\circ \times T_{\Pp^m}$ in $X(\D_\text{toric})$.
\end{lemma}

\begin{proof}
Suppose $g \in I_Y$ and $gt^v \cdot \chi^u \in R(\D_\text{toric})$. Note that $gt^v \in H^0(\Pp^m, \OO(\D_\text{toric}(u)))$. Since the divisor $\D_\text{toric}(u)$ is supported on the complement of $T_{\Pp^m}$, we see that $gt^v$ must be a regular function on $Y^\circ$. Moreover, as $g$ is in $I_Y$, $gt^v$ is in $I_Y$ as well, so $gt^v$ vanishes on $Y^\circ$. Since $Y^\circ$ is dense in $Y$, $gt^v$ vanishes on all of $Y$. It follows then that $gt^v\cdot \chi^u$ is in $I_\D$.

For the other inclusion, start by taking any \(M\)-homogeneous element of $I_\D$, say $gt^v\cdot\chi^u$. Then, we must show that \(g \in I\). Note \(gt^v\) is a global section of $\OO(\D_\text{toric}(u))$, so it is regular on $T_{\Pp^m}$ since $\D_\text{toric}(u)$ is torus-invariant. By definition of $I_\D$, we know that the restrction of $gt^v$ to $Y$ is 0. On $Y^\circ$, $t^v$ is invertible; hence, the restriction of $g$ to $Y^\circ$ vanishes as well, so $g \in I_Y$.
\end{proof}

Lemma \ref{first gen set} gives some geometric intuition; however, the given list of generators is not finite. The next pass at the ideal distinguishes a finite generating set for $I_\D$. Begin with a generating set $\mathcal{G}$ for $I_Y \cap \Cc[t_1,\dotsc, t_m]$ which contains a reduced Gr\"obner basis with respect to any term order which refines total degree (see \cite{Cox-Little-O'Shea} for a definition of a Gr\"obner basis) and is minimal in the sense that if \(g\) is in \(\mathcal{G}\), then \(g/t_i\) is not in $I_Y \cap \Cc[t_1,\dotsc, t_m]$ for each \(i\). For each \(g\) in \(\mathcal{G}\), let \(M_g := \deg(g)\) and define the \textit{degree polyhedron} \(P_g \subseteq \Rr^m \times N_\Rr\) by the following inequalities.
\begin{align*}
    v_i &\geq -\Delta_i(u) \;\;\; \text{for all } i = 1,\dotsc, m\\
    \sum_{i=1}^m v_i &\leq \Delta_0(u) - M_g
\end{align*}
Let \(\mathcal{P}_g\) be a \(\delta^\vee\)-module generating set for the polyhedron \(P_g\). This can be computed by finding a Hilbert basis of the positive hull of \(P_g \times \{1\}\), taking those members of the basis at height 1, and projecting them back to \(\Zz^m \times M\).

\begin{proposition}\label{finite generating set}
The set 
\[ 
\mathcal{S} = \bigcup_{g \in \mathcal{G}} \left\{ gt^v\cdot \chi^u \st (v,u) \in \mathcal{P}_g\right\}
\]
is a finite generating set for \(I_\D\).
\end{proposition}

Before proving Proposition \ref{finite generating set}, we need the following lemma. It shows that for $g\in \mathcal{G}$, the degree polyhedron of $g$ is precisely the set of lattice points $(v,u)$ so that $gt^v \cdot \chi^u \in R(\D_\text{toric})$.

\begin{lemma}\label{degree polyhedron lemma}
Let \(\mathcal{G}\) be as above, and let \(g = \sum_{\alpha\in \Zz^m} c_\alpha t^\alpha \in \mathcal{G}\). Then \((v,u) \in P_g\) if and only if for each nonzero \(c_\alpha\), \(t^{\alpha + v} \cdot \chi^u \in R(\D_\text{toric})\). In particular, \(gt^v\cdot\chi^u \in I_\D\) if and only if \((v,u) \in P_g\). 
\end{lemma}

\begin{proof}
Suppose \((v,u) \in P_g\) and \(c_\alpha \neq 0\). Since \(\alpha_i\) is non-negative and \((v,u) \in P_g\), for each \(i =1 ,\dotsc, m\), the following inequalities hold:
\begin{align*}
v_i + \alpha_i &\geq  v_i \\	
					&\geq -\lfloor\Delta_i(u)\rfloor.
\end{align*}
Also, since \(\sum_i \alpha_i \leq M_g\), we get the other inequality:
\begin{align*}
\sum_{i=1}^m v_i + \alpha_i &\leq (\lfloor\Delta_0(u)\rfloor - M_g) + M_g \\
					&= \lfloor\Delta_0(u)\rfloor.
\end{align*}
Thus, \((v+\alpha, u)\) is in \(\delta^\vee\), i.e. \(t^{v+\alpha}\cdot \chi^u\) is in \(R(\D_\text{toric})\).

On the other hand, suppose that \((v + \alpha, u)\) is in \(\delta^\vee\) for each non-zero \(c_\alpha\). By the assumptions on $\mathcal{G}$, for each \(i\), there must exist some \(\alpha_j\) so that the \(i^\text{th}\) coordinate is zero. Thus,
\[ 
    v_i = v_i + (\alpha_j)_i \geq -\lfloor\Delta_i(u)\rfloor. 
\]
Also, there is some \(\alpha_0\) so that sum of its coordinates is exactly \(M_g\); therefore,
\[
    \sum_{i=1}^m v_i + (\alpha_0)_i \leq \lfloor\Delta_0(u)\rfloor \iff \sum_{i=1}^r v_i \leq \lfloor\Delta_0(u)\rfloor - M_g.
\]
Thus, \((v,u)\) is in \(P_g\).
\end{proof}

\begin{proof}[Proof of \ref{finite generating set}]
Let \(u\) be in \(\sigma^\vee \cap M\). The graded pieces of \(R(\D_\text{toric})\) can be rewritten as follows:
\[
    R(\D_\text{toric})_u = g_u \cdot \chi^u \cdot \Cc[t_1,\dotsc, t_r]_{\leq d_u}
\]
where 
\[
    d_u := \sum_{i=0}^m \lfloor \Delta_i(u) \rfloor  \text{ and } g_u = t_1^{-\lfloor \Delta_1(u) \rfloor} \dots t_r^{-\lfloor \Delta_r(u) \rfloor}.
\]

Lemma \ref{first gen set} implies that \(\mathcal{S}\) is contained in \(I_\D\), so it suffices to show that \(\mathcal{S}\) generates the ideal. The ideal is \(M\)-graded, so it actually suffices to show that each homogeneous element of the ideal is generated by \(\mathcal{S}\). To this end, let \(h \cdot g_u \cdot \chi^u\) be an element of \((I_\D)_u\) where \(h\) is in \(I_Y \cap \Cc[t_1, \dotsc, t_m]_{\leq d_u}\). Since \(\mathcal{G}\) contains a Gr\"obner basis for \(I_Y \cap \Cc[t_1,\dotsc, t_m]\) with respect to a term order which refines degree, \(h\) decomposes as

\[
    h = \sum_{g\in \mathcal{G}} c_g g \text{ with } c_g g\in \Cc[t_1,\dotsc, t_m]_{\leq d_u}.
\] 

This reduces the problem to showing that \(\mathcal{S}\) generates all homogeneous elements of the form \(g t^v \cdot \chi^u\) where \(g\) is a member of \(\mathcal{G}\) and \(gt^v \cdot \chi^u\) is in \(R(\D_\text{toric})\). By Lemma \ref{degree polyhedron lemma}, \((v,u)\) is in \(P_g\). Since \(\mathcal{P}_g\) generates \(P_g\) as a \(\delta^\vee\)-module, one can write
\[
(v,u) = (v',u') + (v'',u'') \text{ with } (v',u') \in \mathcal{P}_g \text{ and } (v'',u'') \in \delta^\vee.
\]
It follows that
\[
gt^v \cdot \chi^u = (gt^{v'}\cdot \chi^{u'}) (t^{v''} \cdot \chi^{u''})
\]
and \(gt^{v'}\cdot \chi^{u'}\) is in \(\mathcal{S}\) and \(t^{v''} \cdot \chi^{u''}\) is in \(R(\D_\text{toric})\).
\end{proof}

\subsection{Tropicalizations of Semi-Canonical Embeddings}
We let $J_\D$ denote the vanishing of ideal of $X_\D$ in $\Cc[x_h ~|~ h \in \mathcal{H}]$, $I_\D$ is the vanishing ideal of $X(\D)$ in $\Cc[\delta^\vee \cap (\Zz^m \times M)]$, and $I_Y$ denotes the vanishing ideal of $Y^\circ$. In this subsection, we will describe the tropicalization of a semi-canonical embedding and state a criterion for when a prime cone in $\Trop(I_Y)$ lifts to a prime cone in $\Trop(J_\D)$. 

\begin{remark}\label{rem-cone lifting}
By Lemma \ref{first gen set}, $X(\D) \cap (T_{\Pp^m} \times T_N) = Y^\circ \times T_N$. It follows that $\Trop(I_\D) = \Trop(I_Y) \times N_\Rr$.
\end{remark}

\begin{lemma}\label{tropicalization}
\(\Trop(J_\D)\) is the image of \(\Trop(I_Y) \times N_\Rr\) under the injective linear map
\[
\Rr^m \times N_\Rr \to \Rr^{\#\mathcal{H}}
\]
given by 
\[
v \mapsto (\dotsc, \langle u,v \rangle, \dotsc)_{u \in \mathcal{H}}.
\]
\end{lemma}

\begin{proof}
 We know that \(X(\D)\cap (\Cc^\times)^{\#\mathcal{H}}\) is the image of \(Y^\circ \times T_N  \subseteq T_{\Pp^m} \times T_N\) under the monomial map \(\phi : X(\D_\text{toric}) \hookrightarrow \Aa^{\#\mathcal{H}}\). Now \(\Trop(J_\D) = \Trop(\phi)(\Trop(I_Y) \times N_\Rr)\) by \cite[Corollary 3.2.13]{Maclagan-Sturmfels}. The tropicalization of \(\phi\) is exactly the linear map given above. The proof is complete once we show that it is injective. 
 
 To this end, suppose $v$ is in the kernel. This implies that $v$ pairs to 0 with every Hilbert basis element of $\delta^\vee \cap (\Zz^m \times M)$, so $v$ lies in $(\delta^\vee)^\perp$ which is the lineality space of $\delta$. Since $\D_\text{toric}$ is a proper polyhedral divisor, the dimension of $X(\D_\text{toric})$ is the sum of dimension of the base and the dimension of $T_N$ which is $m + \mathrm{rk}(M)$. On the other hand, the dimension of $X(\D_\text{toric})$ is the Krull dimension of $R(\D_\text{toric})$ which is $\dim (\delta^\vee)$. Thus, we see that 
 \begin{align*}
 \dim((\delta^\vee)^\perp)  &= \mathrm{rk}(\Zz^m \times M) - \dim(\delta^\vee) \\
                            &= m + \mathrm{rk}(M) - \dim (\delta^\vee) \\
                            &= \dim(X(\D_\text{toric})) - \dim (\delta^\vee) \\
                            &= 0.
 \end{align*}
 It follows that $\delta$ is pointed, so the kernel is trivial.
\end{proof}

\begin{lemma}
Let \(w \in\Trop(J_\D)\). Using the linear map above, we view \(w\) as a weight on the monomials of \(\Cc[\Zz^m \times M]\) or just \(\Cc[\Zz^m]\). Define 
\[
Y_w = \overline{V( \mathrm{in}_w (I_Y))} \subseteq \Pp^m.
\]
Then \(Y_w\) intersects the torus of \(\Pp^m\).
\end{lemma}

\begin{proof}
Let $w' \in \Trop(J_\D)$, and let $w \in \Trop(I_\D) = \Trop(I_Y) \times N_\Rr$ be the unique weight that maps to $w'$. We need to show that \(\mathrm{in}_{w}(I_Y)\) does not contain a monomial. If it did, then there would be some \(f \in I_Y\) so that \(\mathrm{in}_w(f) = t^\alpha\), but then 
\[
\mathrm{in}_w(ft^v\cdot\chi^u) = t^{\alpha + v} \cdot \chi^u \in \mathrm{in}_w (I_\D).
\] 
This, however, contradicts that \(w \in \Trop(I_\D)\).
\end{proof}

Note that \(Y_w\) is not the vanishing locus of the initial ideal of \(\mathcal{I}(Y)\) in \(\Cc[x_0,\dotsc,x_m]\), but rather, it is found by first computing the initial ideal of $I_Y$ which lies in the Laurent polynomial ring \(\Cc[t_1^\pm,\dotsc, t_m^\pm]\) and then homogenizing.  This removes any components of the degeneration that lie in the complement of the torus. 

For any \(w\in \Trop(J_\D)\), one can compute \(\mathrm{in}_w(J_\D)\) and ask when it is prime. On the other hand, one could try to apply the semi-canonical construction with $Y_w$ replacing $Y$ and ask whether that ideal is prime. The second question is more feasible, because it is prime whenever \(Y_w\) is irreducible as it will be the vanishing ideal for a (possibly non-normal) irreducible $T$-variety. Therefore, a sufficient condition for \(\mathrm{in}_w(J_\D)\) to be prime is to require that these two operations commute and for \(Y_w\) to be irreducible. The subsequent theorem states when this can happen, and in such cases, one sees that a prime cone of \(\Trop(I_Y)\) lifts to a prime cone of \(\Trop(J_\D)\).

\begin{theorem}\label{cone-lifting}
Let \(\D\) be a polyhedral divisor on $Y$. Moreover, assume that $Y \subseteq \Pp^m$ is projectively normal and $\D$ is a polyhedral pullback of a $T_{\Pp^m}$-invariant polyhedral divisor. Fix \(w \in \Trop(J_\D)\) so the following hold.
\begin{enumerate}
    \item \(Y_w\) is irreducible.
    \item There is a generating set \(\mathcal{G} \subseteq I(Y)\) as in \ref{degree polyhedron lemma}, and the \(w\)-initial forms of \(\mathcal{G}\) must generate \(\mathrm{in}_w(I_Y)\).
    \item For any \(g \in \mathcal{G}\), the degree polyhedra of $g$ and $\mathrm{in}_w(g)$ are equal.
\end{enumerate}
In this situation, \(\mathrm{in}_w(J_\D)\) is \(J_{\D_w}\) where \(\D_w\) is the polyhedral divisor \(\sum_{i=0}^m \Delta_i \otimes H_i\vert_{Y_w}\). In particular, this ideal is prime, so the cone in $\Trop(I_Y)$ which contains the image of $w$ lifts to a prime cone in $\Trop(J_\D)$.
\end{theorem}

\begin{proof}
Let \(w \in \Trop(J_\D)\). By assumption, \(Y_w \subseteq \Pp^m\) is irreducible and intersects \(T_{\Pp^m}\), so the semi-canonical construction applies to \(Y_w\) with the polyhedral pullback of \(\mathfrak{D}_\text{toric}\) along the embedding of \(Y_w \subseteq \Pp^m\). This new polyhedral divisor on $Y_w$ by $\D_w$. Even if \(Y_w \subseteq \Pp^m\) is not projectively normal, the resulting ideal will still be prime as it is the kernel of a ring homomorphism whose codomain is an integral domain.

We will show that \(\mathrm{in}_w(J_\D) = J_{\D_w}\). To this end, it is enough to show equality of the ideals in the semi-group algebra \(\Cc[\delta^\vee \cap (\Zz^m \times M)]\), so we will show 
\[
    \mathrm{in}_w(I_\D) = I_{\D_w}.
\]
By assumption, the generating set \(\mathcal{G}\) of \(I_Y\subseteq \Cc[t_1^\pm ,\dots, t_r^\pm]\) has been chosen so that \(\mathrm{in}_w(\mathcal{G})\) generates \(I_{Y_w}\). The key is that for all \(g \in \mathcal{G}\), the polyhedra \(P_g\) and \(P_{\mathrm{in}_w(g)}\) are equal. These polyhedra are all given by the following inequalities.

\[
\begin{cases}
v_i \geq -\Delta_i(u) & i = 1,\dotsc, m \\
\sum_{i=1}^m v_i \leq \Delta_0(u) - \deg(g) &  \\
\end{cases}
\]

Since both ideals are \(M\)-graded, we take a homogeneous element \(f t^v\cdot\chi^u \in I_\D\) and its initial form,
\[
\mathrm{in}_w(f t^v\cdot \chi^u) = \mathrm{in}_w(f)t^v \cdot \chi^u.
\]
Note that \(\mathrm{in}_w(f)\) must be in \(I_{Y_w}\), so by Lemma \ref{first gen set}, we see that \(\mathrm{in}_w(I_\D) \subseteq I_{\D_w}\). 

On the other hand, since \(\mathrm{in}_w(\mathcal{G})\) generates \(\mathrm{in}_w(I_{Y})\) and \(P_g = P_{\mathrm{in}_w(g)}\), Lemma \ref{degree polyhedron lemma} implies that \(I_{\D_w}\) is generated by elements of the form 
\[
\mathrm{in}_w(g) t^v \cdot \chi^u\;\; \text{with} \;\; g \in \mathcal{G} \text{ and } (v,u) \in \mathcal{P}_g
\]
which are all \(w\)-initial forms of the generators of \(I_\D\); hence, \(I_{\D_w} \subseteq \mathrm{in}_w(I_\D)\).
\end{proof}

This subsection ends with an example illustrating how the final condition of the theorem is necessary.

\begin{example}
Consider the elliptic curve \(Y = V(x_0 x_2^2 - x_1^3 - x_0^2x_1)\), let \(\iota\) be the inclusion of \(Y\) into \(\Pp^2\), and take the polyhedral divisor on \(Y\), 
\[
\mathfrak{D} = \left[ \frac{6}{5}, \infty \right) \otimes \iota^*(V(x_0)) + \left[ -\frac{1}{2}, \infty \right) \otimes \iota^*(V(x_1)) + \left[ -\frac{2}{3}, \infty \right) \otimes \iota^*(V(x_2)).
\]
Clearly, $\D$ is a polyhedral pullback of a proper polyhedral divisor on $\Pp^2$. In this case \(\mathcal{G} = \{t_2^2 - t_1^3 - t_1 \} \subseteq \Cc[t_1, t_2]\). Note that there are two initial ideals which are prime; however, only one of them gives a prime cone for \(\Trop(J_\D)\). Either by hand or using Macualay2, one sees that 
\[
J_\D = \langle X_1^2 X_3^{10} - X_2^6 X_3^5 + X_1^6 \rangle \subseteq \Cc[X_1,X_2,X_3].
\]
There are three maximal cones in the tropical variety, and one can compute there corresponding initial ideals using the following weights.
\begin{align*}
    w_1 &= (-46, -31, -18)\\
    w_2 &= (10, 7, 4)\\
    w_3 &= (26,17,10)
\end{align*}
These were found by taking the rays of the inner normal fan of the Newton polygon for \(f = t_2^2 - t_1^3 - t_1\), appending a zero in the last entry, and passing them through the linear map in Lemma \ref{tropicalization}.
\begin{center}
 \begin{tabular}{||c || l | l | l ||} 
\hline
 $i$ & $Y_{w_i}$ & $\mathrm{in}_{w_i}(J_\D)$ & $J_{\D_{w_i}}$ \\ 
 \hline
 1 & $V(x_2^2 - x_0 x_1)$ & $\langle X_1^2 X_3^{10} - X_2^6 X_3^5 \rangle$ & $\langle X_1^2 X_3^5 - X_2^6 \rangle$ \\

 2 & $V(x_1^2 - x_0^2)$ & $\langle  X_1^2 X_3^{10} - X_1^6\rangle$ & $\langle X_3^{10} - X_1^4 \rangle$ \\

 3 & $V(x_0 x_2^2 - x_1^3)$ & $\langle X_2^6 X_3^5 - X_1^6 \rangle$ & $\langle X_2^6 X_3^5 - X_1^6 \rangle$  \\
 \hline

\end{tabular}
\end{center}
One sees that \(\mathrm{in}_{w_3}(J(Y)) = J(Y_{w_3})\) is prime. The other two pairs are not equal, nor is \(\mathrm{in}_{w_i}(J(Y))\) prime for \(i=1,2\). This is happening for two reasons. The first weight does not give equal ideals since the degree polyhedron changes after taking the initial form with respect to $w_1$, and the second weight does not work either since the initial ideal of $I_Y$ with respect to $w_2$ is not prime.
\end{example}

\begin{example}
By \cite{Speyer-Sturmfels}, $\mathrm{Gr}(2,n)$ is well-poised under its Pl\"ucker embedding. Moreover, it is not difficult to see that the Pl\"ucker relations satisfy the conditions in Theorem \ref{cone-lifting}. Hence, any proper torus-invariant polyhedral divisor $\D_\text{toric}$ on $\Pp^{\binom{n}{2} - 1}$ produces a well-poised affine $T$-variety by pulling $\D_\text{toric}$ back to $\mathrm{Gr}(2,n)$.
\end{example}

\section{Applications of Theorem \ref{cone-lifting}} \label{sec:examples}

In this section, we apply Theorems \ref{thm-quotients of wellpoised} and \ref{cone-lifting} to show two classes of varieties have well-poised embeddings. We begin with an analysis of arrangement varieties which have been studied by \cite{Ilten-Manon} and \cite{Hausen-Hische-Wrobel}. Recall that an arrangement variety $X$ is a $T$-variety whose base is a projective space. We prove Theorem \ref{thm-main}, i.e. that an affine arrangement variety is well-poised under its semi-canonical embedding. This is done in Theorem \ref{genArrVar}. As discussed in Section \ref{prelim-valuations}, this implies each cone in the tropicalization induces a valuation $\mathfrak{v}$ on the coordinate ring of $X$. When the hyperplane arrangement is general and $\mathfrak{v}$ is coming from a maximal cone of $\Trop(X)$, we argue that the valuation has full rank, and we describe the corresponding value semigroups and Newton-Okounkov cones. As a consequence, the semi-canonical generators form a Khovanskii basis with respect to any such valuation.

In the latter subsection, we introduce hypertoric varieties which have been widely studied by \cite{Hausel-Sturmfels}, \cite{Kutler}, \cite{Bielawski-Dancer}, \cite{Proudfoot}, and others. These varieties are often thought of as hyperk{\"a}hler or quaternionic analogues of toric varieties. As we shall see, these varieties can be constructed as GIT quotients of affine arrangement varieties; hence, Theorems \ref{thm-quotients of wellpoised} and \ref{genArrVar} imply that hypertoric varieties have well-poised embeddings.

%%%%%%%%%%%%%%%%%%%%%%
% Arrangement Varieties
%%%%%%%%%%%%%%%%%%%%%%

\subsection{Arrangement Varieties}\label{genArrVar}

We begin by showing affine arrangement varieties always have well-poised embeddings. Then for each \(M\)-homogeneous valuation coming from the a facet of tropical variety, we describe the value semigroups and Newton-Okoounkov cones. We also give a necessary and sufficient condition for when these semigroups are saturated.

An affine arrangement variety is an affine \(T\)-variety whose base is \(\Pp^c\) and whose polyhedral divisor is of the form 
\[
\mathfrak{D} = \sum_{i=0}^m \Delta_i \otimes V(\ell_i)
\]
where each \(\ell_i \in \Cc[x_0,\dotsc, x_c]_1\), the hyperplanes \(V(\ell_i)\) form a full rank hyperplane arrangement, and \(\sum_{i=1}^m \Delta_i \subsetneq \sigma\). Now, embed \(Y = \Pp^c\) into \(\Pp^m\) via the map 
\[
P \mapsto [\ell_0(P): \dots : \ell_m(P)]
\]
Using this map, one can find a semi-canonical embedding of \(X(\mathfrak{D})\). Indeed, denoting the map above by \(\varphi\), the polyhedral divisor \(\mathfrak{D}\) is the polyhedral pullback of the torus-invariant polyhedral divisor \(\D_\text{toric} = \sum_{i=0}^m \Delta_i \otimes H_i\) on \(\Pp^m\). Note that it is possible that not all the hyperplanes in the arrangement are distinct. In this case, the polyhedral pullback $\D_\text{toric}$ is not proper. It still defines an affine $T$-variety, but it may fail to be normal.

\begin{theorem}\label{general arrangement varieties are well-poised}
Semi-canonical embeddings of affine arrangement varieties are well-poised.
\end{theorem}

\begin{proof}
The base is well-poised because \(Y_w \isom \Pp^c\) for any \(w \in \Trop(I_Y)\). We can choose \(\mathcal{G}\) to be the circuits of \(I(Y)\) which always satisfies the conditions in Theorem \ref{cone-lifting}. Finally, by Remark \ref{rem-cone lifting} and Lemma \ref{tropicalization}, every cone in \(\Trop(J_\D)\) arises as a lift of a cone from \(\Trop(I_Y)\), so the semi-canonical embedding is well-poised.
\end{proof}

Next, we want to describe the value semigroups and Newton-Okounkov cones of semi-canonically embedded affine \textit{general} arrangement varieties for the valuations coming from the maximal cones of \(\Trop(J_\D)\). While a similar analysis could be done for hyperplane arrangements which are not in general linear position, we restrict to the general case for simplicity.

\begin{theorem}\label{value-semigroup}
Let $\D = \sum_{i=0}^m \Delta_i \otimes V(\ell_i)$ where each $\ell_i$ is a linear form in $\Cc[x_0,\dotsc,x_c]$ and $\sum_i \Delta_i \subsetneq \sigma$. Moreover, assume that the hyperlane arrangement is of full rank and in general linear position. Let \(w \in \mathrm{relint}(\tau)\) where \(\tau\) is a maximal cone in the tropicalized linear space \(\Trop(L)\) where \(L\) is the image of $\Pp^c$ in $\Pp^m$ under the map 
\[
    P \mapsto [\ell_0(P): \dots : \ell_m(P)].
\]
The value semigroup, \(S(R(\D), \mathfrak{v}_w)\), is isomorphic to the following semigroup:
\[
S_w := \left\{ (v,u) \in \Zz^c \times M \st \sum_{i=1}^c v_i \leq \sum_{j \in \mathcal{I}} \lfloor \Delta_j(u)\rfloor \text{ and } v_j \geq -\lfloor \Delta_{i_j}(u) \rfloor \text{ for all } 1\leq j\leq c \right\}
\]
where  \(\mathcal{I}\) is the set of indices where \(w\) is minimized and \(\{i_1,\dotsc,i_c\}\) is  \(\{0,\dotsc, m\} \setminus \mathcal{I}\).
\end{theorem}

\begin{remark}
The valuations in the statement above are living in $\Trop(L)$. Remark \ref{rem-cone lifting} and Lemma \ref{tropicalization} show that $\mathfrak{v}_w$ is the weight valuation for the image of $(w,0)$ under the linear map from \ref{tropicalization}.
\end{remark}

\begin{proof}
The base of \(X(\mathfrak{D})\) is the linear subspace of \(L \subseteq \Pp^m\) which corresponds to the general hyperplane arrangement in \(\Pp^c\) given in its polyhedral divisor. The tropicalization of \(L\) is the Bergman fan of the uniform matroid \(U_{c+1, m+1}\), so its maximal faces are indexed by subsets of \(\{0,\dotsc,m\}\) of size \(c\) and are of the form
\[
\tau_{i_1\dots i_c} = \Qq_{\geq0}\{e_{i_1},\dotsc,e_{i_c}\}
\]
where \(\{e_i ~|~ 1 \leq i \leq m\}\) is the usual standard basis and \(e_0 = - \sum_{i=1}^m e_i\). 

If \(w \in \text{relint}(\tau_{i_1\dots i_c})\), \(L_w\) is cut out by binomials supported on \(\mathcal{I}\). Observe that for any \(i\), \(V(x_i) \cap L_w\), is a torus-invariant divisor on \(L_w\) since \(L_w\) is a toric subvariety of \(\Pp^m\) and \(L_w\) intersects \(T_{\Pp^m}\). 

For any two distinct indices $i_j,i_k$ not in $\mathcal{I}$, one sees that \(V(x_{i_j}) \cap L_w \neq V(x_{i_k}) \cap L_w\) since \(L_w\) is cut out by binomials that are supported on the variables $\{x_i ~|~ i \in \mathcal{I}\}$; therefore, the divisors \(D_j = V(x_{i_j})\cap L_w\) constitute \(c\) distinct torus-invariant divisors of \(L_w\). 

For any $k \in \mathcal{I}$, the divisors, \(V(x_k) \cap L_w\), are also torus invariant on \(L_w\), and these are distinct from the divisors listed in the previous paragraph. Thus, these must all be the same since there are only \(c+1\) torus-invariant divisors on \(L_w\). This last divisor is denoted $D_0$

After a change of coordinates, we may assume that the divisors $D_0,\dotsc,D_c$ are the standard coordinate hyperplanes on $\Pp^c$. Pulling \(\mathfrak{D}_\text{toric}\) back to $L_w$ yields the following algebra.
\[
\bigoplus_{u \in \sigma^\vee \cap M} H^0\left(\Pp^c, \OO\left(\left(\sum_{j \in \mathcal{I}} \lfloor \Delta_j(u) \rfloor\right) D_0 + \sum_{j=1}^c \lfloor \Delta_{i_j}(u)\rfloor D_j \right) \right) \cdot \chi^u
\]
For each $u \in \sigma^\vee \cap M$, set $\D_\mathcal{I}(u)$ to be 
\[
\left(\sum_{j \in \mathcal{I}} \lfloor \Delta_j(u) \rfloor\right) D_0 + \sum_{j=1}^c \lfloor \Delta_{i_j}(u)\rfloor D_j.
\]
For each $u$, $\D_\mathcal{I}(u)$ is a torus-invariant divisor on $\Pp^c$ whose global sections are spanned by monomials. In fact, $t^v \in H^0(\Pp^c, \OO(\D_\mathcal{I}(u))$ if and only if $(v,u) \in S_w$. Thus, the algebra above is \(\Cc[S_w]\) which completes the proof.
\end{proof}

\begin{remark}
As a consequence of Theorem \ref{genArrVar}, each maximal cone of $\Trop(J_\D)$ gives rise to a full rank valuation on $R(\D)$, and the semi-canonical generators of $R(\D)$ form a Khovanskii basis for any of these valuations. Moreover, the proof of \ref{genArrVar} can be adapted to compute the value semigroups for any affine arrangement variety, so even in the non-generic case, the semi-canonical generators are a Khovanskii basis.
\end{remark}

The Newton-Okounkov cone, \(C(R(\D), \mathfrak{v}_w)\), is the closure of the positive hull of \(S_w\). Using Theorem \ref{value-semigroup}, the Newton-Okounkov cones can be described as follows.

\begin{corollary}\label{cor-NOKcone}
The Newton-Okounkov cone, \(C(R(\D),\mathfrak{v}_w)\), is isomorphic to 
\[
\left\{ (v,u) \in \Rr^c \times M_\Rr \st \sum_{i=1}^c v_i \leq \sum_{j \in \mathcal{I}} \Delta_j(u) \text{ and } v_j \geq - \Delta_{i_j}(u) \text{ for all } 1\leq j\leq c \right\}
\]
where \(\mathcal{I}\) is the set where \(w\) is minimized and \(\{i_1,\dotsc,i_c\}\) is \(\{0,\dotsc,m\} \setminus \mathcal{I}\).
\end{corollary}

With this characterization of of the value semigroups, we can ask what properties they have. We start with a necessary and sufficient condition on the polyhedral divisor for the value semigroup to be saturated, i.e. for $X(\D_w)$ to be a normal toric variety. Then as normal toric varieties are Cohen-Macaulay and the Cohen-Macaulay property is an open condition in flat families, we derive a sufficient (but not necessary) condition for \(X(\D)\) to be Cohen-Macaulay. To this end, we make the following definition.

\begin{definition}
Let \(\{\Delta_i\}\) be a collection of rational polyhedra with common tail cone \(\sigma\). The collection of polyhedra, \(\{\Delta_i\}\), is {\it admissable} if for each \(u \in \sigma^\vee \cap M\) at most one of \(\text{face}(\Delta_i,u)\) is non-integral. Here \(\text{face}(P,u)\) is the face of \(P\) on which the linear functional \(u\) is minimized. 
\end{definition}

\begin{remark}
Collections of admissable polyhedra were studied in \cite{Ilten-Suss} to study test configurations on rational complexity 1 \(T\)-varieties.
\end{remark}

\begin{proposition}\label{saturated-value-semigroup}
Let \(w\), \(S_w\), and \(\mathcal{I}\) be as in \ref{value-semigroup}. Then \(S_w\) is saturated if and only if \(\{\Delta_i\}_{i\in\mathcal{I}}\) is admissable.
\end{proposition}

\begin{proof}
Suppose \(\{\Delta_i\}_{i \in \mathcal{I}}\) is not admissable. In this case, there is some \(u \in \sigma^\vee \cap M\) so that 
\[
    \left\lfloor \sum_{i \in \mathcal{I}} \Delta_i(u) \right\rfloor - \sum_{i \in \mathcal{I}} \lfloor \Delta_i(u)\rfloor \geq 1.
\]
Therefore, there exists some \((v,u) \notin S_w\), but 
\begin{align*}
    v_j &\geq - \lfloor\Delta_{i_j}(u) \rfloor \\
    \sum_{i=1}^c v_i &\leq \left\lfloor \sum_{i\in\mathcal{I}} \Delta_i(u) \right\rfloor.
\end{align*}
There exists a \(k\in \Nn\) so that \(\Delta_i(ku)\) is integral for each \(i \in \mathcal{I}\), which forces 
\[
    \left\lfloor \sum_{i} \Delta_i(ku) \right\rfloor = \sum_{i} \lfloor \Delta_{i} (ku) \rfloor.
\]
It follows that \(k(v,u) \in S_w\), so \(S_w\) is not saturated.

Conversely, suppose that \(\{\Delta_i\}_{i \in \mathcal{I}}\) is admissable. Then for each \(u \in \sigma^\vee \cap M\), 
\[
    \sum_{i\in\mathcal{I}} \lfloor \Delta_i(u) \rfloor = \left\lfloor \sum_{i\in\mathcal{I}} \Delta_i(u) \right\rfloor.
\]
Let \(\Delta_\mathcal{I}\) be the Minkowski sum of all \(\Delta_i\) with \(i \in\mathcal{I}\). The right-hand side of the equation above is the equal to \(\lfloor \Delta_\mathcal{I} (u) \rfloor\). Applying the semi-canonical construction to the polyhedral divisor,
\[
\mathfrak{D}' =  \Delta_{\mathcal{I}} \otimes V(y_0) + \Delta_{i_1} \otimes V(y_1) + \dots + \Delta_{i_c} \otimes V(y_c),
\]
yields \(\Cc[S_w]\) as the coordinate ring of \(X(\mathfrak{D}')\). As \(\mathfrak{D}'\) is a proper polyhedral divisor on \(\Pp^c\), this ring must be normal, so \(S_w\) is saturated.
\end{proof}

\begin{corollary}\label{isCM} Let \(\mathfrak{D} = \sum_{i=0}^m \Delta_i \otimes V(\ell_i)\) be the polyhedral divisor for the affine general arrangement variety, \(X(\mathfrak{D})\). If there is some \(\mathcal{I} \subseteq \{0,\dotsc,m\}\) with \(|\mathcal{I}| = m - c + 1\) so that \(\{\Delta_i\}_{i \in \mathcal{I}}\) is admissable, then \(X(\mathfrak{D})\) is Cohen-Macaulay. 
\end{corollary}

While Corollary \ref{isCM} can be used to see if an arrangement variety is Cohen-Macaulay, the next example illustrates that the admissability condition does not always detect the Cohen-Macaulay property.

\begin{example}
Consider the duVal singularity, $E_8=V(x^2 + y^3 + z^5) \subseteq \Aa^3$. This is a complexity 1 general arrangement variety whose polyhedral divisor on $\Pp^1$ is given by 
\[
    \D = \left[ \frac{6}{5}, \infty \right) \otimes \{0\} + \left[ -\frac{1}{2}, \infty \right) \otimes \{1\} + \left[ -\frac{2}{3}, \infty \right) \otimes \{\infty\}.
\]
One can check (using Macaulay2 say) that this embedding in $\Aa^3$ is semi-canonical. Since $\D$ is a proper polyhedral divisor, we know that this variety is normal. Since it is a normal surface, $E_8$ is Cohen-Macaulay; however, any choice of two of the polyhedral coefficients is not admissable. One can also see that the three toric degenerations of $E_8$ are all non-normal toric varieties. 
\end{example}

\subsection{Hypertoric Varieties}\label{sec-hypertoric}
Hypertoric varieties are quaternionic analogues of toric varieties, and as we will see below, they are integrally related to arrangement varieties at least when they are defined over the complex numbers. In this section, we will review how to construct semiprojective hypertoric varieties as GIT quotients of a subvariety of the cotagent bundle \(T^*\Aa^d\) (see \cite{Hausel-Sturmfels} or \cite[Remark 5.1]{Kutler}); we will refer to this subvariety of $T^* \Aa^d$ as the \textit{total space} of the hypertoric variety. Moreover, we will show that the total space is a semi-canonically embedded arrangement variety (Proposition \ref{prop-disjoint sqfree monomial maps and linear ideals}). Then Theorems \ref{thm-quotients of wellpoised} and \ref{general arrangement varieties are well-poised} imply hypertoric varieties have well-poised presentations.

The hypertoric varieties in question arise from hyperplane arrangements. First, we fix some notation. Let \(N\) and \(M\) be mutually dual finite rank lattices. Fix sequences \(a = (a_1,\dotsc, a_d) \in N^d\) and \(r = (r_1,\dotsc, r_d) \in \Zz^d\). Then, define \(\mathcal{A}(a,r)\) to be the hyperplane arrangement \(\{H_{a_i,r_i} \st 1\leq i \leq d\}\) where each \(H_{a_i, r_i} = \{u \in M_\Rr \st \langle a_i, u\rangle = r_i\}\). Given any hyperplane arrangement \(\mathcal{A} = \mathcal{A}(a,r)\), we denote its {\it centralization} by \(\mathcal{A}_0 = \mathcal{A}(a,0)\), and we get a linear subspace \(L_{\mathcal{A}_0}\) of \(\Aa^d\) which is cut out by all linear forms \(\sum_{i=1}^d c_i x_i\) such that \(\sum_{i=1}^d c_i a_i = 0\).

Now, given a full rank hyperplane arrangement, \(\mathcal{A}(a,r)\), the following diagram commutes and the rows and third column are exact.

\begin{center}
    \begin{tikzcd}
    0 \arrow[r] & \Lambda \arrow[r, "i"] \arrow[d, "\mathrm{id}"] & \Zz^d \arrow[r, "a"] \arrow[d, "\nabla"] & N \arrow[r] \arrow[d]& 0 \\
    0 \arrow[r]& \Lambda \arrow[r,"\nabla \circ i"]& \Zz^d \oplus \Zz^d \arrow[r, "\widetilde{a}"]& \widetilde{N}\arrow[r]\arrow[d,"p"] & 0 \\
               &                                   &                              &         \Zz^d \arrow[d]               & \\
               &                                   &                            &                 0          &
    \end{tikzcd}
\end{center}

\noindent The map \(a : \Zz^d \to N\) is defined by mapping \(e_i \mapsto a_i\), \(\nabla : \Zz^d \to \Zz^d \oplus \Zz^d\) is defined by \(e_i \mapsto e_i^+ - e_i^-\) where the superscripts indicate which side of the direct sum the basis vector lives on, and \(p : \widetilde{N} \to \Zz^d\) (the cokernel of the antidiagonal embedding of \(N \to \widetilde{N}\)) is defined by sending \(\overline{e_i}^\pm \mapsto e_i\). 

Note that \(\Zz^d\oplus\Zz^d\) is the co-character lattice of \(\Aa^{2d} \isom T^*\Aa^d\), and the composition \((p \circ \widetilde{a})_\Rr\) maps \(\Rr^{2d}_{\geq 0}\) onto \(\Rr^d_{\geq 0}\). Thus, there is a surjective map of toric varieties \(T^* \Aa^d \to \Aa^d\). Secondly, there is a linear projection \(\Aa^d \to \Aa^d/ L_{\mathcal{A}_0}\). The composition of these two morphisms gives us a moment map \(\mu_\mathcal{A} : T^*\Aa^d \to \Aa^d/L_{\mathcal{A}_0}\) for the Hamiltonian action of the torus \(G_\mathcal{A} := \Spec(\Cc[\Lambda^*])\) on \(T^*\Aa^d\) (where \(\Lambda^*\) is the dual of \(\Lambda)\)). 

\begin{definition}\label{hypertoric-def}\cite[Remark 5.1]{Kutler}
Given a hyperplane arrangement \(\mathcal{A} := \mathcal{A}(a,r)\), we define the hypertoric variety \(\mathfrak{M}_\mathcal{A}\) to be the GIT quotient \(\mu_\mathcal{A}^{-1}(0) \sslash_\alpha G_\mathcal{A}\) where \(\mu_\mathcal{A}\) and \(G_\mathcal{A}\) are as above and \(\alpha = i^*(a)\). The preimage of the origin under the moment map, $\mu_\mathcal{A}^{-1}(0)$, is referred to as the \textit{total space} of $\mathfrak{M}_\AA$.
\end{definition}

Let \(S = \Cc[x_1,\dotsc, x_d, y_1,\dotsc, y_d]\) be the homogeneous coordinate ring for \(T^*\Aa^d\) so \(\deg(x_i) = 1\) and \(\deg(y_i) = -1\) for all \(i = 1,\dotsc d\), let \(R = \Cc[t_1,\dotsc,t_d]\) be the coordinate ring for \(\Aa^d\) and let \(I\) be the vanishing ideal for \(L_{\mathcal{A}_0}\). The moment map is defined by the ring homomorphism 
\begin{align*}
    \varphi : R &\to S \\ t_i &\mapsto x_i y_i,
\end{align*}

\noindent and therefore, \(\mu_\mathcal{A}^{-1}(0) = V(\varphi(I))\). The following proposition shows that ideals of this form are semi-canonical presentations of certain arrangement varieties.

\begin{proposition}\label{prop-disjoint sqfree monomial maps and linear ideals}
Let \(I \subseteq \Cc[x_1,\dotsc,x_n]\) be a linear ideal which does not contain any monomials; fix a sequence of positive integers \((i_1, \dotsc, i_n)\) where each \(i_k \geq 2\); and set \(N = \sum_{j=1}^n i_j\). Define the ring homomorphism $\psi$ as follows.
\begin{align*}
\psi : \Cc[x_1,\dotsc,x_n] &\to \Cc[y_{1,1},\dotsc,y_{1,i_1}, \dotsc, y_{n,1},\dotsc, y_{n, i_n}] \\
x_k &\mapsto \prod_{j=1}^{i_k} y_{k,j}
\end{align*}
Let \(J = \psi(I)\). Then \(X = V(J)\) is a well-poised arrangement variety, and the presentation given is semi-canonical.
\end{proposition}

\begin{proof}
The first steps are to identify the $T$-action on $X$, compute the Chow quotient of $X$ by $T$, and to give a description of its polyhedral divisor. We will follow the procedure laid out in \cite[\S 11]{Altman-Hausen}. There is a \(T = (\Cc^\times)^{N - n + 1}\)-action on \(\Aa^N\) which restricts to an effective action on \(X\). The \(T\)-action is described by the following matrix. 
\[
F = 
\begin{pmatrix}
F_{i_1} & 0 & \dotsc & 0 &\ell_{i_1} \\
0       & F_{i_2} & \dotsc & 0 & \ell_{i_2} \\
\vdots & \vdots & \ddots & \vdots & \vdots \\ 
0      & 0 & \hdots & F_{i_n} & \ell_{i_n}
\end{pmatrix}
\]
Here \(F_k\) is in \(\Zz^{k \times (k - 1)}\) and $\ell_k$ is in $\Zz^{k \times 1}$, and they are defined as follows.
\[
(F_k)_{ij} = \begin{cases}
-1 & \text{if }i = j \\
1 & \text{if }i = j + 1 \\
0 & \text{else}
\end{cases},
\;\;\;
(\ell_k)_{i1} = \begin{cases}
1 & \text{if }i = k \\
0 & \text{else}
\end{cases}
\]
The $T$-action as follows: given \(t \in T\) and \(p \in X\)(or $\Aa^N$)
\[
 t\cdot p := (\dotsc, t^{\alpha_{j,k}} p_{j,k}, \dotsc) \in X(\text{or }\Aa^N)
\]
where \(\alpha_{j,k}\) is the \((i_1 + \dots + i_{j-1} + k)^\text{th}\) row of \(F\). This induces a split exact sequence of co-character lattices which can be used to compute the normalized Chow quotient, \(Y := X\sslash T\), and to compute the polyhedral divisor \(\mathfrak{D}\) on \(Y\) corresponding to \(X\). By identifying the co-character lattices of \(T\), \((\Cc^\times)^N\), and \((\Cc^\times)^N \sslash T\), by \(\Zz^{N- n + 1}, \Zz^N,\) and \(\Zz^{n-1}\), respectively, one arrives at the split exact sequence below.

\begin{center}
    \begin{tikzcd}
    0 \ar{r}& \Zz^{N - n + 1} \ar[bend left = 33]{r}{F} & \Zz^{N} \ar[bend left = 33]{l}{s} \ar{r}{P} & \Zz^{n - 1} \ar{r} & 0
    \end{tikzcd}
\end{center}

\noindent The matrix \(F\) is already defined, and the other two matrices are as follows:

\[
P = 
\begin{blockarray}{ccccccccccccc}
\begin{block}{\BAmulticolumn{3}{c}\BAmulticolumn{3}{c}\BAmulticolumn{3}{c}\BAmulticolumn{1}{c}\BAmulticolumn{3}{c}}
 i_1 & i_2 & i_3 & & i_n \\
 \overbrace{\qquad \qquad \qquad} & \overbrace{\qquad \qquad \qquad } & \overbrace{\qquad \qquad \qquad} & \hdots & \overbrace{\qquad \qquad \qquad} \\
\end{block}
\begin{block}{(ccccccccccccc)}
-1 & \hdots & -1 & 1 & \hdots & 1 & 0 & \hdots & 0 & \hdots & 0 & \hdots & 0 \\
-1 & \hdots & -1 & 0 & \hdots & 0 & 1 & \hdots & 1 & \hdots & 0 & \hdots & 0 \\
\vdots & \ddots & \vdots & \vdots & \ddots & \vdots & \vdots & \ddots & \vdots & \ddots & \vdots & \ddots & \vdots \\
-1 & \hdots & -1 & 0 & \hdots & 0 & 0 & \hdots & 0 & \hdots & 1 & \hdots & 1 \\
\end{block} 
\end{blockarray}.
\]

\[
s = 
\begin{pmatrix}
s_{i_1} & 0 & \dotsc & 0 \\
0       & s_{i_2} & \dotsc & 0 \\
\vdots & \vdots & \ddots & \vdots  \\ 
0      & 0 & \hdots & s_{i_n} \\
\mathbb{1}& 0 & \hdots & 0
\end{pmatrix}
\]
where $s_k \in \Zz^{(k-1)\times k}$ is defined entrywise $(s_k)_{ij} = -1$ if $i\geq j$ and 0 otherwise, and $\mathbb{1} = (1,\dotsc,1) \in \Zz^{1 \times i_1}$.

One checks that \(sF = I_{N-n+1}\) and the sequence is exact. The normalized Chow quotient, \(\mathbb{A}^N \sslash T\), is given by the toric variety \(Z(\Sigma)\) where \(\Sigma\) coarsest fan which refines all cones \(P(\tau)\) where \(\tau\) is a face of the positive orthant \(\Qq_{\geq 0}^N\). In this case, \(\Sigma\) is the standard fan for \(\Pp^{n-1}\); therefore, \(Y\) is the closure of the image of \(X \cap (\Cc^\times)^N\) under the monomial map \(\mu_P : (\Cc^\times)^N \to \Pp^{n-1}\) induced by \(P\), 
\[(y_{1,1},\dotsc, y_{n,i_n}) \mapsto \left[\psi(x_1) : \psi(x_2) : \dotsc : \psi(x_n) \right],\]
Then \(Y\) can be seen to be \(\Pp(V(I)) \subseteq \Pp^{n-1}\). It follows that \(X\) is an arrangement variety whose polyhedral divisor is 
\[
    \mathfrak{D} := \sum_{k=0}^{n-1} \Delta_k \otimes D_k
\]
where \(D_k := H_k\vert_Y\) is a hyperplane. If we let $\{e_0,\dotsc,e_{n-1}\}$ be the standard ray generators for the fan of $\Pp^{n-1}$, the polyhedron \(\Delta_k\) is \(s(\Qq_{\geq 0}^N \cap P^{-1}(e_k))\). The common tail cone among these polyhedra is $\sigma = s(\Qq_{\geq 0})$. Upon analyzing each \(\Delta_k\) further, one sees that each is the Minkowski sum of the polytope $\Pi_k$ and the cone $\sigma$ where $\Pi_k$ is the convex hull of the \((1 + \sum_{j=1}^{k} i_j)^\text{th}\) through $(\sum_{j=1}^{k+1} i_j)^\text{th}$ columns of \(s\). For $k \geq 1$, each $\Delta_k = \sigma$ as $\Pi_k \subset \sigma$ and $0$ is a vertex of $\Pi_k$. Thus, the polyhedral divisor $\D$ simplifies to the expression below.
\[
    \D = \Delta_0 \otimes D_0
\]

Recall that a semi-canonical embedding of $X(\D)$ depends on a projectively normal embedding of the base in some projective space so that each Cartier divisor $D$ appearing in $\D$ is a pullback of a coordinate hyperplane. As $X \isom X(\D)$, Theorem \ref{general arrangement varieties are well-poised} implies any semi-canonical embedding of \(X\) is well-poised. It remains to show that the presentation given is semi-canonical with respect to the embedding $Y = \Pp(V(I)) \subseteq \Pp^{n-1}$.

By Proposition \ref{Ambient Toric Variety Prop}, one can embed \(X(\mathfrak{D})\) into a normal toric variety \(X(\D_\text{toric})\) whose rational polyhedral cone \(\delta\) is the positive hull of the columns of the following matrix.
\[
A:= \left( 
\begin{array}{c|c}
   \widehat{P}  & I_{n-1} \\
  \hline
   \widehat{s} & 0
\end{array}
\right) \in \Zz^{N \times N}
\]
Here $\widehat{s}$ is the square $(N-n+1) \times (N-n+1)$ matrix obtained from $s$ by deleting all of the columns which are identically zero, and $\widehat{P} \in \Zz^{(n-1) \times (N-n+1)}$ is defined as follows. 
\[
    \widehat{P}_{ij} = \begin{cases}
    -1 & \text{if }j \leq i_1 \\
    0 & \text{otherwise}
    \end{cases}
\]
We claim that $\delta$ is smooth. If this is the case, then $X(\D_\text{toric})$ is isomorphic to $\Aa^N$ under the map given by sending the $y_{i,j}$'s to the ray generators of $\delta^\vee$, and it follows that the semi-canonical embedding is the same as the embedding given as both presentations will agree on the dense torus of $\Aa^N$.

To this end, it suffices to show that $\det(A) = \pm 1$. Note that the inverse of $\widehat{s}$ is $\widehat{F}$ which is obtained from $F$ by deleting the rows of $F$ which correspond to the zero columns of $s$. Consider the following $N\times N$ integral matrix.
\[
    B := \left( 
        \begin{array}{c|c}
        0   & \widehat{F} \\
        \hline
        I_{n-1} & 0
        \end{array}
        \right) 
\]
Then the product $AB$ is 
\[
\left( 
    \begin{array}{c|c}
        \widehat{P}   & I_{n-1} \\
        \hline
        \widehat{s} & 0
    \end{array}
\right)
\left( 
    \begin{array}{c|c}
        0   & \widehat{F} \\
        \hline
        I_{n-1} & 0
    \end{array}
\right)
= 
\left( 
    \begin{array}{c|c}
        I_{n-1}  & \widehat{P}\widehat{F} \\
        \hline
        0 & I_{N-n+1}
    \end{array}
\right)
= I_N
\]
The product has determinant 1, so $\det(A) = \pm 1$, and it follows that $\delta$ is smooth.
\end{proof}

\begin{corollary}\label{cor-total space}
The total space, $\mu_\AA^{-1}(0)$, is well-poised.
\end{corollary}
\begin{proof}
This follows by the Proposition \ref{prop-disjoint sqfree monomial maps and linear ideals} and Theorem \ref{genArrVar}.
\end{proof}

\begin{corollary}\label{cor-hypertoric}
Given a full rank hyperplane arrangement $\mathcal{A}$, the corresponding hypertoric variety $\mathfrak{M}_\mathcal{A}$ has a well-poised embedding.
\end{corollary}

\begin{proof}
This follows by Theorem \ref{thm-quotients of wellpoised} and Corollary \ref{cor-total space}.
\end{proof}

\begin{remark}
Recall that a tropicalization is \textit{faithful} if there is a continuous section $s : \Trop(X) \to X^\text{an}$ of the tropicalization map where $X^\text{an}$ is the Berkovich analytification.
In \cite[Theorem 6.3]{Kutler}, Kutler showed that hypertoric varieties are faithfully tropicalized by their Lawrence embeddings by analyzing the structure of $\Trop(\mathfrak{M}_\AA)$ and appealing to a result of Gubler, Rabinoff, and Werner \cite[Theorem 8.14]{Gubler-Rabinoff-Werner}. As remarked in the introduction, the fact that $\mathfrak{M}_\AA$ has a well-poised embedding provides an alternate proof that hypertoric varieties have faithful tropicalizations.
\end{remark}

\nocite{*}
\bibliographystyle{alpha}
\bibliography{main}

\end{document}